\documentclass{amsart}
\usepackage{psfrag}
\usepackage{subcaption}
\usepackage{color}
\usepackage{tikz}
\usetikzlibrary{matrix,arrows}
\usepackage{graphicx,graphics}
\usepackage{amssymb,amsfonts,amsmath,amstext,amsthm,amscd,verbatim,enumerate,mathrsfs}
\usepackage[normalem]{ulem}
\usepackage[T1]{fontenc}
\usepackage{hyperref}

\DeclareMathOperator{\dist}{dist}
\DeclareMathOperator{\diam}{diam}

\begin{document}

\newtheorem{theorem}{Theorem}[section]
\newtheorem{result}[theorem]{Result}
\newtheorem{fact}[theorem]{Fact}
\newtheorem{conjecture}[theorem]{Conjecture}
\newtheorem{lemma}[theorem]{Lemma}
\newtheorem{proposition}[theorem]{Proposition}
\newtheorem{corollary}[theorem]{Corollary}
\newtheorem{facts}[theorem]{Facts}
\newtheorem{props}[theorem]{Properties}
\newtheorem*{thmA}{Theorem A}
\newtheorem{ex}[theorem]{Example}
\theoremstyle{definition}
\newtheorem{definition}[theorem]{Definition}
\newtheorem{remark}[theorem]{Remark}
\newtheorem{example}[theorem]{Example}
\newtheorem*{defna}{Definition}

\newcommand{\notes} {\noindent \textbf{Notes.  }}
\newcommand{\note} {\noindent \textbf{Note.  }}
\newcommand{\defn} {\noindent \textbf{Definition.  }}
\newcommand{\defns} {\noindent \textbf{Definitions.  }}

\renewcommand{\int}{\operatorname{int}}

\newcommand{\x}{{\bf x}}
\renewcommand{\d}{{\delta}}
\newcommand{\g}{{\gamma}}
\newcommand{\z}{{\bf z}}
\newcommand{\e}{{\epsilon}}
\newcommand{\B}{{\bf b}}
\newcommand{\V}{{\bf v}}
\newcommand{\T}{\mathbb{T}}
\renewcommand{\S}{\mathbb{S}}
\newcommand{\Z}{\mathbb{Z}}
\newcommand{\Hp}{\mathbb{H}}
\newcommand{\D}{\mathbb{D}}
\newcommand{\R}{\mathbb{R}}
\newcommand{\N}{\mathbb{N}}
\newcommand{\W}{\mathcal{W}}
\renewcommand{\B}{\mathbb{B}}
\newcommand{\C}{\mathbb{C}}
\newcommand{\ft}{\widetilde{f}}
\newcommand{\dt}{{\mathrm{det }\;}}
 \newcommand{\adj}{{\mathrm{adj}\;}}
 \newcommand{\0}{{\bf O}}
 \newcommand{\av}{\arrowvert}
 \newcommand{\zbar}{\overline{z}}
 \newcommand{\xbar}{\overline{X}}
 \newcommand{\htt}{\widetilde{h}}
\newcommand{\ty}{\mathcal{T}}
\renewcommand\Re{\operatorname{Re}}
\renewcommand\Im{\operatorname{Im}}
\newcommand{\tr}{\operatorname{Tr}}
\newcommand\interior{\operatorname{int}}
\renewcommand{\skew}{\operatorname{skew}}

\newcommand{\ds}{\displaystyle}
\numberwithin{equation}{section}

\renewcommand{\theenumi}{(\roman{enumi})}
\renewcommand{\labelenumi}{\theenumi}

\newcommand{\alastair}[1]{{\scriptsize \color{red}\textbf{Alastair's note:} #1 \color{black}\normalsize}}

\newif\ifcomment

%\commenttrue

\title{Genus $g$ Cantor sets and germane Julia sets}

\author{A. Fletcher}
\email{afletcher@niu.edu}
\address{Department of Mathematical Sciences, Northern Illinois University, Dekalb, IL 60115, USA}

\author{D. Stoertz}
\email{stoert1@stolaf.edu}
\address{Department of Mathematics, Statistics, and Computer Science, St.\@ Olaf College, Northfield, MN 55057, USA}

\author{V. Vellis}
\email{vvellis@utk.edu}
\address{Department of Mathematics, The University of Tennessee, Knoxville, TN 37966}


\thanks{V.~Vellis was partially supported by NSF DMS grants 1952510 and 2154918.}

\date\today

\begin{abstract}
The primary aim of this paper is to give topological obstructions to Cantor sets in $\R^3$ being Julia sets of uniformly quasiregular mappings. Our main tool is the genus of a Cantor set. We give a new construction of a genus $g$ Cantor set, the first for which the local genus is $g$ at every point, and then show that this Cantor set can be realized as the Julia set of a uniformly quasiregular mapping. These are the first such Cantor Julia sets constructed for $g\geq 3$. We then turn to our dynamical applications and show that every Cantor Julia set of a hyperbolic uniformly quasiregular map has a finite genus $g$; that a given local genus in a Cantor Julia set must occur on a dense subset of the Julia set; and that there do exist Cantor Julia sets where the local genus is non-constant.
\end{abstract}

\maketitle

\section{Introduction}

It is well-known that the Julia set $J(f)$ of a rational map $f$ can be a Cantor set. The simplest examples arise for quadratic polynomials $z^2+c$ when $c$ is not in the Mandelbrot set. It is also well known that every Cantor set embedded in $\overline{\R^2}$ has a defining sequence consisting of topological disks, that is, every such Cantor set arises as an infinite intersection of a collection of nested disks, see \cite{Moise}.

The goal of the current paper is to study topological properties of Julia sets of uniformly quasiregular mappings (henceforth denoted by UQR mappings) in $\overline{\R^3}$ and, in particular, when they are Cantor sets, what sort of defining sequences they can have. UQR mappings provide the setting for the closest counterpart to complex dynamics in $\overline{\R^3}$ and, more generally, higher real dimensions. We will, however, stay in dimension three in this paper
as this provides the setting to consider the genus of a Cantor set as introduced by \v{Z}eljko \cite{Z1} based on the notion of defining sequences from Armentrout \cite{Ar}.

The first examples of UQR mappings constructed by Iwaniec and Martin \cite{IM} have a Cantor set as the Julia set. Moreover, although this was not of concern to the authors, from their construction it is evident that the Julia set is a tame Cantor set. This means that the Cantor set can be mapped via an ambient homeomorphism of $\overline{\R}^3$ onto the standard ternary Cantor set contained in a line. Equivalently, this means the Cantor set has a defining sequence consisting of topological $3$-balls. Moreover, such a Cantor set $X$ is then said to have genus zero, written $g(X) = 0$.

If a Cantor set is not tame, then it is called wild. The standard example of a wild Cantor set in $\R^3$ is Antoine's necklace. The first named author and Wu \cite{FW} constructed a UQR map for which the Julia set is an Antoine's necklace that has genus $1$.
More recently, the first and second named authors \cite{FS} showed via a more intricate construction that there exist UQR mappings whose Julia sets are genus $2$ Cantor sets. 

The first main aim of the current paper is to give a general construction which will apply to all genera. This will necessitate a new topological construction since, as far as the authors are aware, the only construction of genus $g$ Cantor sets for all $g$ are given by \v{Z}eljko \cite{Z1} and this construction cannot yield Julia sets, as will be seen via Corollary \ref{cor:1} below. The local genus $g_x(X)$ of a Cantor set $X$ at $x\in X$ describes the genus of handlebodies required in a defining sequence in any neighborhood of $x$. \v{Z}eljko's construction has local genus one except at one point. Our first main result reads as follows.

\begin{theorem}
\label{thm:1}
For each $g\in\N$ there exists a UQR map $f_g:\overline{\R^3} \to \overline{\R^3}$ for which the Julia set $J(f_g)$ is a Cantor set of genus $g$ and, moreover, for each $x\in J(f_g)$, the local genus $g_x(J(f_g)) = g$.
\end{theorem}

We remark that the genus $1$ case of Theorem \ref{thm:1} recovers Antoine's necklace, whereas the genus $2$ case is substantially different from the construction in \cite{FS}. For all higher genera, Theorem \ref{thm:1} provides a new construction. This construction is necessarily highly intricate as it needs to be amenable to our dynamical applications.

Next, we turn to topological obstructions for Cantor sets in $\overline{\R^3}$ being Julia sets based on the genus. It is an important theme in dynamics to  give geometric or topological restrictions on the Julia set, once a toplogical type has been fixed. The first named author and Nicks \cite{FN11} showed that the Julia set of a UQR mapping in $\overline{\R^n}$ is uniformly perfect, that is, ring domains which separate points of the Julia set cannot be too thick. As a counterpart to this result, it was shown by the first and third named authors \cite{FV}, that if the Julia set of a hyperbolic UQR mapping in $\overline{\R^n}$ is totally disconnected, then it is uniformly disconnected. Here, a uniformly quasiregular mapping is hyperbolic if the Julia set does not meet the closure of the post-branch set. Roughly speaking, this result says that ring domains separating points of the Julia set cannot be forced to be too thin. 

The above results place geometric conditions on which Cantor sets can be Julia sets. Our second main result in this paper places a topological restriction on which Cantor sets can be Julia sets.

\begin{theorem}
\label{thm:3}
Let $f:\overline{\R^3} \to \overline{\R^3}$ be a hyperbolic UQR map for which $J(f)$ is a Cantor set. Then there exists $g \in \N \cup \{ 0 \}$ such that the genus of $J(f)$ is $g$.
\end{theorem}

There do exist Cantor sets of infinite genus, see \cite[Theorem 5]{Z1}, and so Theorem \ref{thm:3} rules these out as possibilities for Julia sets of hyperbolic UQR maps. In particular, an even stronger version of Theorem \ref{thm:3} is true: if a Julia set of a hyperbolic UQR map is a Cantor set, then it has a defining sequence which consists of at most finitely many (up to similarities) different handlebodies; see Lemma \ref{lem:finite defining sequence}. This lemma leads to a quasiregular uniformization of Cantor sets in $\R^3$ which may be of independent interest; see Appendix \ref{ap:QRunif}.

We recall that the backwards orbit of $x$ is 
\[ O^-(x) = \{ y : f^m(y) = x \text{ for some }m\in \N \},\]
and the grand orbit is
\[ GO(x) = \{ y : f^{m_1}(y) = f^{m_2}(x) \text{ for some } m_1,m_2\in \N\}.\]
Our next result is on the local genus of points in the Julia set.

\begin{theorem}
\label{thm:4}
Let $f:\overline{\R^3} \to \overline{\R^3}$ be a hyperbolic UQR map for which $J(f)$ is a Cantor set. If the local genus $g_x(J(f)) = g \in \N \cup \{ 0 \}$, then $g_y(J(f)) = g$ for every $y$ in the grand orbit $GO(x)$.
\end{theorem}

As the backwards orbit of a point in $J(f)$ is dense in $J(f)$, we immediately have the following corollary.

\begin{corollary}
\label{cor:1}
Let $f:\overline{\R^3} \to \overline{\R^3}$ be a hyperbolic UQR map with $J(f)$ a Cantor set. Suppose there exists $x\in J(f)$ with $g_x(J(f)) = g$. Then the set of points in $J(f)$ for which the local genus is $g$ is dense in $J(f)$.
\end{corollary}

This result places further severe restrictions on which Cantor sets can be Julia sets of hyperbolic UQR maps. The constructions in \cite[Theorem 5]{Z1} which yield Cantor sets of genus $g \in \N$ have the property that there is a special point $x\in X$ for which $g_x(X) = g$ and $g_y(X) = 1$ for all other points $y\in X \setminus \{ x\}$. Corollary \ref{cor:1} then implies that these Cantor sets cannot be Julia sets.

Since the examples of Julia sets in \cite{FS,FW} have constant local genus, it is natural to ask if this is always the case for Julia sets which are Cantor sets. Our final result shows that this is not the case.

\begin{theorem}
\label{thm:5}
Let $g\geq 1$.
There exists a hyperbolic UQR map $f:\overline{\R^3} \to \overline{\R^3}$ such that $J(f)$ is a Cantor set of genus $g$, and there exist points with local genus $g$ and other points with local genus $0$.
\end{theorem}

It would be interesting to know whether any finite collection of non-negative integers can be realized as the local genera of a Cantor Julia set.

\medskip

The paper is organized as follows. In Section \ref{sec:prelims} we recall some preliminary material on UQR maps and the genus of Cantor sets. In Section \ref{sec:construction} we construct a Cantor set $X_g$ for each $g\geq 1$. In Section \ref{sec:genus}, we prove that $X_g$ has genus $g$ and local genus $g$ at each point. In Section \ref{sec:julia} we complete the proof of Theorem \ref{thm:1} by constructing a UQR map with Julia set equal to $X_g$. In Section \ref{sec:genusjulia} we prove Theorem \ref{thm:3} and Theorem \ref{thm:4}. Finally, in Section \ref{sec:genusg0} we construct an example that proves Theorem \ref{thm:5}.

\subsection*{Acknowledgements}
We thank the referee for their valuable comments which have greatly improved the exposition of the paper.

\section{Preliminaries}\label{sec:prelims}

 We denote by $\overline{\R^n}$ the one point compactification of $\R^n$.

\subsection{Uniformly quasiregular mappings}

A continuous map $f: \R^n \to \R^n$ is called \emph{quasiregular} if $f$ belongs to the Sobolev space $W^{1,n}_{\text{loc}}(\R^n)$ and if there exists some $K\geq 1$ such that 
\begin{equation}\label{eq:QR1} 
|f'(x)|^n \leq K J_f(x) \qquad \text{for a.e. $x\in \R^n$}.
\end{equation}
Here $J_f$ denotes the Jacobian of $f$ at $x\in \R^n$ and $|f'(x)|$ the operator norm. 
If $f$ is quasiregular, then there exists $K'\geq 1$ such that
\begin{equation}\label{eq:QR2} 
J_f(x) \leq K' \min _{|h|=1}  |f'(x)(h)|^n \qquad \text{ for a.e. $x\in \R^n$}.
\end{equation}
The {\it maximal dilatation} $K(f)$ of a quasiregular map $f$ is the smallest $K$ that satisfies both equations \eqref{eq:QR1} and \eqref{eq:QR2}. The maximal dilatation $K(f)$ can be informally thought of as a quantity describing how much distortion $f$ has. The closer $K(f) \in [1,\infty) $ is to $1$, the closer $f$ is to a conformal map. If $K(f) \leq K$, then we say that $f$ is $K$-quasiregular. See Rickman's monograph \cite{Rickman} for a complete exposition on quasiregular mappings. 

Quasiregular mappings can be defined at infinity and also take on the value infinity. To do this, if $A:\overline{\R^n} \to \overline{\R^n}$ is a M\"obius map with $A(\infty) = 0$, then we require $f\circ A^{-1}$ or $A\circ f$ respectively to be quasiregular via the definition above.

{\it Bounded length distortion} maps, BLD for short, are a sub-class of quasiregular maps for which the finite length curves are mapped to curves of finite length, with uniform control on the length distortion. BLD maps are bi-Lipschitz around non-branch points (see below) and can be viewed as an intermediate step between quasiregular and bi-Lipschitz mappings. In a sense, BLD maps are to bi-Lipschitz maps what quasiregular maps are to quasiconformal maps.

The composition $f\circ g$ of two quasiregular mappings $f$ and $g$ is again quasiregular, but typically the maximal dilatation goes up. We say that $f:\overline{\R^n} \to \overline{\R^n}$ is uniformly quasiregular, abbreviated to UQR, if the maximal dilatations of all the iterates of $f$ are uniformly bounded above.

For a UQR map, the definitions of the Julia set and Fatou set are identical to those in complex dynamics: the Fatou set $F(f)$ is the domain of local normality of the family of the iterates and the Julia set $J(f)$ is the complement.

The \emph{branch set} $\mathcal{B}(f)$ of a UQR map $f: \overline{\R^n} \to \overline{\R^n}$ is the closed set of points in $\overline{\R^n}$ where $f$ does not define a local homeomorphism. The \emph{post-branch set} of non-injective UQR map $f$ is
\[ \mathcal{P}(f)=\overline{\{ f^m(\mathcal{B}(f) ) : m\geq 0 \}}.\]
The map $f$ is called hyperbolic if $J(f) \cap \mathcal{P}(f)$ is empty.

We will need the following result regarding injective restrictions of hyperbolic UQR maps near the Julia set.

\begin{lemma}[\cite{FV}, Lemma 3.3 and the proof of Theorem 3.4]
\label{lem:1}
Suppose that $n\geq 2$ and $f:\overline{\R^n} \to \overline{\R^n}$ is a hyperbolic UQR map {with $\infty \not\in J(f)$}. There exists $r_1 > 0$ such that if $x\in J(f)$, then $f$ is injective on $B(x,r_1)$. Moreover, there exists $N\in \N$ such that if $U$ is an $r_1$-neighbourhood of $J(f)$, then $\overline{f^{-N}(U)} \subset U$.
\end{lemma}

\subsection{Cantor sets and genus}

Recall that a Cantor set is any metric space homeomorphic to the usual Cantor ternary set. Two Cantor sets $E_1, E_2 \subset \overline{\R^n}$ are equivalently embedded (or ambiently homeomorphic) if there exists a homeomorphism $\psi:\overline{\R^n} \to \overline{\R^n}$ such that $\psi(E_1) = E_2$. If the Cantor set $E$ is equivalently embedded to the usual Cantor ternary set in a line, then $E$ is called tame. A Cantor set which is not tame is called wild.
We often assume that $\infty \notin E$ so we may consider $E\subset \R^n$.

Other examples of Cantor sets in $\overline{\R^n}$ are typically defined in terms of a similar construction to {that of the usual Cantor ternary set}, using an intersection of nested unions of compact $n$-manifolds with boundary. For Cantor sets in $\R^3$, the idea of defining sequences goes back to Armentrout \cite{Ar}. This can be easily generalized to Cantor sets in $\overline{\R^3}$ by applying a M\"obius map so as to move the Cantor set to $\R^3$.

\begin{definition}
	A {\it defining sequence} for a Cantor set $E\subset \R^3$ is a sequence $(M_i)$ of compact 3-manifolds with boundary such that
	\begin{enumerate}[(i)]
		\item each $M_i$ consists of disjoint polyhedral cubes with handles,
		
		\item {$M_{i+1}$ is contained in the interior of  $M_i$} for each $i$, and
		
		\item $E = \bigcap_i M_i$.
	\end{enumerate}
	We denote the set of all defining sequences for $E$ by $\mathcal{D}(E)$.
\end{definition}

{By \cite[Theorem 8]{Ar}, for every Cantor set in $\R^3$, there exists at least one defining sequence.}

{If $\mathcal{C}$ is a topological cube with handles, denote the number of handles of $\mathcal{C}$ by $g(\mathcal{C})$. For a disjoint union of cubes with handles $M = \sqcup_{i\in I} \mathcal{C}_i$, we set $g(M) = \sup \{ g(\mathcal{C}_i) : i\in I \}$.}
The genus of a Cantor set was introduced by \v{Z}eljko, see \cite[p. 350]{Z1}.

\begin{definition}\label{def:genus}
	Let $(M_i)$ be a defining sequence for the Cantor set $E \subset \R^3$. Define
		\[g(E;(M_i)) = \sup\{g(M_i) : i \geq 0 \}.\]
	Then we define the {\it genus} of the Cantor set $E$ as
		\[g(E) = \inf\{g(E;(M_i)) : (M_i) \in \mathcal{D}(E)\}. \]
	Now let $x \in E$. For each $i$, denote by $M_i^x$ the unique component of $M_i$ containing $x$. Similar to above, define
		\[g_x(E;(M_i)) = \sup\{g(M_i^x) : i \geq 0 \}.\]
	Then we define the {\it local genus} of $E$ at the point $x$ as
		\[g_x(E) = \inf\{g_x(E;(M_i)) : (M_i) \in \mathcal{D}(E)\}. \]
\end{definition}

\section{Construction of a genus $g$ self-similar Cantor set}\label{sec:construction}

\subsection{Some sequences}\label{sec:seq}
We start by defining a ``folding'' sequence that will help us keep track of various folding maps that will be required later on. For each $n\in\N$ and each $k\in\{1,\dots,n\}$ define $a_{n,k} \in \N$ with $a_{1,1}=1$ and for $m\in\N$
\begin{align*} 
a_{2m,k} &= 
\begin{cases}
a_{m,k} &\text{if $1\leq k \leq m$}\\
a_{m,2m-k+1} &\text{if $m+1\leq k \leq 2m$},
\end{cases}\\
a_{2m+1,k} &= 
\begin{cases}
a_{m+1,k} &\text{if $1\leq k \leq m$}\\
2a_{m+1,m+1} &\text{if $k = m+1$}\\
a_{m+1,2m-k+2} &\text{if $m+2\leq k \leq 2m+1$}.
\end{cases}
\end{align*}

\begin{lemma}
For all $n\in\N$ and $k\in\{1,\dots,n\}$, we have that $a_{n,n}= a_{n,1} = 1$ and $a_{n,k}\in \{1,2\}$.
\end{lemma}

\begin{proof}
The proof of the claim is by induction on $n$. The claim is clear for $n=1$. Assume now the claim to be true for all integers $n < N$ for some $N\in\N$. 

If $N=2m$, then 
\[ \{ a_{2m,k}  : 1\leq k \leq 2m\} = \{ a_{m,k}  : 1\leq k \leq m\} \subset \{1,2\}\]
while $a_{2m,1} = a_{m,1} = 1$ and $a_{2m,2m} = a_{m,1} = 1$.

If $N=2m+1$, then 
\[ \{ a_{2m+1,k}  : 1\leq k \leq 2m, k\neq m+1\} = \{ a_{m+1,k}  : 1\leq k \leq m\} \subset \{1,2\}\]
and $a_{2m+1,m+1} = 2a_{m+1,m+1} =2$. Moreover, 
\[ a_{2m+1,1} = a_{m+1,1} = 1\quad \text{and}\quad a_{2m+1,2m+1} = a_{m+1,1} = 1. \qedhere\]
\end{proof}

For each $n\in\N$ we define a finite sequence $(c_{n,i})_{i=1}^n$ by  
\[ c_{n,i} =
\begin{cases}
1 &\text{ if $a_{n,i}=1$}\\
3 &\text{ if $a_{n,i}=2$}.
\end{cases}
\]
Set $C_{n,0}=0$ and for each $i\in\{1,\dots,n\}$, set 
\[ C_{n,i} = c_{n,1}+\cdots+c_{n,i}.\] 

For each $g\in\N$ we fix for the rest of the paper an odd square integer $N_g\in\N$ such that
\begin{equation}\label{eq:N} 
N_g\geq 20\sqrt{6}(8C_{g,g}+6).
\end{equation}
This important integer is related to the number of handlebodies required at each stage of the defining sequence of the Cantor set to be constructed.

\subsection{A $g$-ladder}\label{sec:ladder}

Fix $g\in\N$. To ease the notation, for the rest of Section \ref{sec:construction}, we write $N_g=N$, $c_{g,i}=c_i$ and $C_{g,i}=C_i$.

For each $i\in \{1,\dots,g\}$ define the planar simple closed curve
\begin{equation}\label{eq:gamma}
\gamma_{i} = \partial\left([C_{i-1}, C_{i} ]\times[0,1]\right).
\end{equation}
Note that each $\g_i$ is either a translated copy of $\partial [0,1]^2$ or a translated copy of $\partial ([0,3]\times[0,1])$. Define also the planar closed curve
\[ \gamma = \gamma_{1} \cup \cdots \cup  \gamma_{g}.\]
See Figure \ref{fig1} for $\gamma$ in the case that $g=6$. 

\begin{figure}
\includegraphics[width=\textwidth]{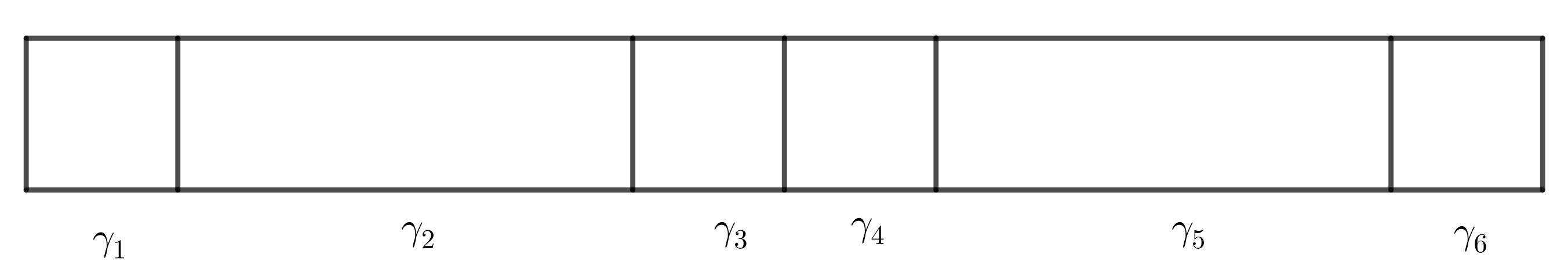}
\caption{The curve $\gamma$ in the case that $g=6$. In this case, $(a_{6,i}) = (1,2,1,1,2,1)$.}
\label{fig1}
\end{figure}

\subsection{A chain of $g$-ladders}\label{sec:second}

For each $i\in\{1,\dots,g\}$ and $j\in \{1,\dots,(2c_i+2)N\}$ we define rescaled copies $\tau_{i,j}$ of $\gamma$. Fix $i\in \{1,\dots,g\}$. Let $x_{i,1}= (C_{i-1}+\frac1{2N},1)$, $x_{i,c_{i}N+1}= (C_{i}-\frac1{2N},1)$, $x_{i,(c_{i}+1)N+1}= (C_{i}-\frac1{2N},0)$, $x_{i,(2c_{i}+1)N+1}= (C_{i-1}+\frac1{2N},0)$, and define also
\begin{align*}
x_{i,j} =
\begin{cases}
({C_{i-1}}+{(j-1)}/N,1),  &\text{ if $j=2,\dots,c_iN$}\\
({C_i},1-(j-c_iN-1)/N), &\text{ if $j=c_iN+2,\dots,(c_i+1)N$}\\
 ({C_i}-(j-(c_i+1)N-1)/N,0), &\text{ if $j=c_i(N+1)+2,\dots,(2c_i+1)N$}\\
({C_{i-1}},(j-(2c_i+1)N-1)/N), &\text{ if $j=(2c_i+1)N+2,\dots,(2c_i+2)N$}.\\
\end{cases}
\end{align*}
Therefore, the points $\{x_{i,1},\dots,x_{i,2(c_i+1)N}\}$ lie on $\g_i$ oriented clockwise.

For each $j\in\{1,\dots,2(c_i+1)N\}$, let $\sigma_{i,j}$ be a line segment in $\R^2$ centered at $x_{i,j}$, of length $8\sqrt{2}(5N)^{-1}$, such that 
\begin{enumerate}
\item $\sigma_{i,j}$ is {perpendicular} to $\gamma_i$ if $j\in\{1,c_iN+1, (c_i+1)N+1, (2c_i+1)N{+1}\}$;
\item $\sigma_{i,j}$ has slope $-1$ if 
\begin{enumerate}
\item $j\in \{2,4,\dots, c_iN-1\}$, 
\item or if $j\in \{c_iN+2,c_iN+4,\dots,(c_i+1)N-1\}$, 
\item or if $j\in \{(c_i+1)N+2,(c_i+1)N+4,\dots,(2c_i+1)N-1\}$, 
\item or if $j\in \{(2c_i+1)N+2,(2c_i+1)N+4,\dots,(2c_i+2)N-1\}$;
\end{enumerate}
\item $\sigma_{i,j}$ has slope $1$ if 
\begin{enumerate}
\item $j\in \{3,5,\dots, c_iN\}$, 
\item or if $j\in \{c_iN+3,c_iN+5,\dots,(c_i+1)N\}$, 
\item or if $j\in \{(c_i+1)N+3,(c_i+1)N+5,\dots,(2c_i+1)N\}$, 
\item or if $j\in \{(2c_i+1)N+3,(2c_i+1)N+5,\dots,(2c_i+2)N\}$.
\end{enumerate}
\end{enumerate}
See Figure \ref{fig2} for the case $c_i=1$ (for simplicity we have set $N=9$).

\begin{figure}[h]
\includegraphics[scale=0.7]{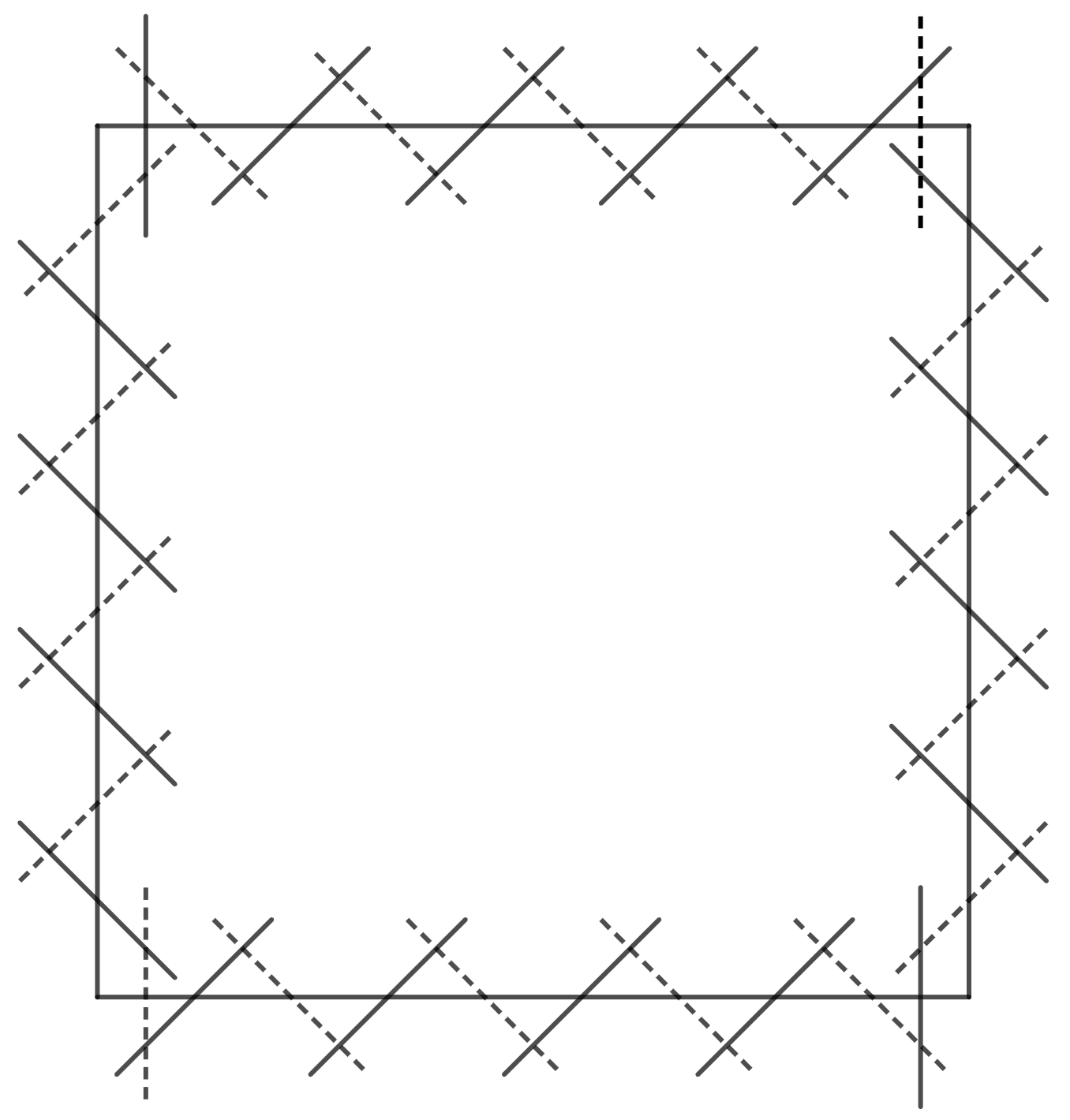}
\caption{The segments $\sigma_{i,j}$ in the case that $c_i=1$ and $N=9$. The segments $\sigma_{i,j}$ with $j$ even, are depicted as dotted segments.}
\label{fig2}
\end{figure}

\begin{lemma}\label{lem:sigma}
If $i,i' \in \{1,\dots,g\}$, $j \in \{1,\dots, 2(c_i+1)N\}$, and $j'\in\{1,\dots,2(c_{i'}+1)N\}$, then $\sigma_{i,j}\cap \sigma_{i',j'} \neq \emptyset$ if and only if one of the following is true:
\begin{enumerate}
\item $i=i'$ and ${j-j'\in\{-1,0,1\}}$ modulo $(2c_i+2)N$ (that is, either $j=j'$, or $x_{i,j}$, $x_{i,j'}$ are consecutive on {$\g_i$});
\item $|i- i'|=1$ and {$x_{i,j} = x_{i',j'}$}. 
\end{enumerate}
Moreover, if $\sigma_{i,j}\cap \sigma_{i',j'} = \emptyset$, then $\dist(\sigma_{i,j},\sigma_{i',j'}) \geq {(5\sqrt{2}N)^{-1}}$.
\end{lemma}

\begin{proof}
To prove the lemma, we consider three possible cases.

\emph{Case 1: $i=i'$}. If $j=j'$, then clearly $x_{i,j}=x_{i',j'}$ and $\sigma_{i,j}\cap \sigma_{i',j'} \neq \emptyset$. We assume for the rest of Case 1 that $j\neq j'$ and consider {five} subcases. 

\emph{Case 1.1: ${j-j'\in\{-1,1\}}$ modulo $(2c_i+2)N$}. It is easy to see by the design of the segments $\sigma_{i,j}$ that $\sigma_{i,j}\cap \sigma_{i',j'} \neq \emptyset$. 

\emph{Case 1.2: ${j-j'\in\{-2,2\}}$ modulo $(2c_i+2)N$ and $x_{i,j},x_{i,j'}$ are both on the left edge or both are in the right edge}. Then $\sigma_{i,j}$ and $\sigma_{i,j'}$ are parallel and $\dist(\sigma_{i,j},\sigma_{i,j'}) = \sqrt{2}/N >(5\sqrt{2}N)^{-1}$. 

\emph{Case 1.3: ${j-j'\in\{-2,2\}}$ modulo $(2c_i+2)N$ and $x_{i,j},x_{i,j'}$ are both on the top edge or both are on the bottom edge}. Without loss of generality, assume that $x_{i,j},x_{i,j'}$ are both on the top edge. Suppose first that $j,j' \not\in \{1,c_iN+1\}$. Then, $\sigma_{i,j}$ and $\sigma_{i,j'}$ are parallel and $\dist(\sigma_{i,j},\sigma_{i,j'}) = \sqrt{2}/N > (5\sqrt{2}N)^{-1}$. Suppose now that $j=1$; the other case is similar. Then, elementary calculations show that 
\[ \dist(\sigma_{i,j},\sigma_{i,j'}) = \tfrac{7}{16}8\sqrt{2}(5N)^{-1} >(5\sqrt{2}N)^{-1}.\] 

\emph{Case 1.4: ${j-j'\in\{-2,2\}}$ modulo $(2c_i+2)N$ and $x_{i,j},x_{i,j'}$ are not on the same edge}. Without loss of generality assume that $x_{i,j}$ is on the left edge and $x_{i,j'}$ is on the top edge. If $j=2(c_i+1)N$ and $j'=2$, then elementary calculations show that
\[ \dist(\sigma_{i,j},\sigma_{i,j'}) = \tfrac18 8\sqrt{2}(5N)^{-1} >(5\sqrt{2}N)^{-1}.\]
If $j=2(c_i+1)N-1$ and {$j'=1$}, then elementary calculations show that
\[ \dist(\sigma_{i,j},\sigma_{i,j'}) = {\sqrt{1965+1280\sqrt{2}}(10N)^{-1}}>(5\sqrt{2}N)^{-1}.\] 

\emph{Case 1.5: ${j-j' \not\in \{-2,-1,0,1,2\}}$ modulo $(2c_i+2)N$}. Note that $|x_{i,j}-x_{i,j'}| \geq \sqrt{5}/N$. Suppose first that $|x_{i,j}-x_{i,j'}| = \sqrt{5}/N$. This is the case where one of the two points $x_{i,j}, x_{i,j'}$ (say $x_{i,j}$) is either on the left or on the right edge, and the other point $x_{i,j'}$ is either on the top or on the bottom edge. In either case, $\sigma_{i,j}$ is parallel to $\sigma_{i,j'}$ and elementary calculations give
\[ \dist(\sigma_{i,j},\sigma_{i,j'}) \geq (\sqrt{2}N)^{-1} > (5\sqrt{2}N)^{-1}.\]
In the case that $|x_{i,j}-x_{i,j'}| > \sqrt{5}/N$ we have that $|x_{i,j}-x_{i,j'}| \geq 3/N$. Fix a point $z\in \sigma_{i,j}$ and a point $z'\in\sigma_{i',j'}$. Then,
\[ |z-z'| \geq |x_{i,j}-x_{i',j'}| - 2(4\sqrt{2}(5N)^{-1}) \geq 3/N - 8\sqrt{2}(5N)^{-1} > (5\sqrt{2}N)^{-1}.\]
Therefore, $\dist(\sigma_{i,j},\sigma_{i,j'}) > (5\sqrt{2}N)^{-1}$.

\emph{Case 2: $|i-i'|\geq 2$}. Fix a point $z\in \sigma_{i,j}$ and a point $z'\in\sigma_{i',j'}$. By the choice of $N$,
\begin{align*} 
|z-z'| \geq |x_{i,j}-x_{i',j'}| - 2(4\sqrt{2}(5N)^{-1}) &\geq \dist(\gamma_i,\gamma_{i'}) - 8\sqrt{2}(5N)^{-1}\\ 
&= 1- 8\sqrt{2}(5N)^{-1}\\ 
&> (5\sqrt{2}N)^{-1}.
\end{align*}
Therefore, $\dist(\sigma_{i,j},\sigma_{i,j'}) > (5\sqrt{2}N)^{-1}$.

\emph{Case 3: $|i-i'| =1$}. Without loss of generality, we assume that $i' = i+1$. There are four subcases to consider.

\emph{Case 3.1: $x_{i,j}$ is on the top edge of $\g_i$ and $x_{i+1,j'}$ is on the top edge of $\g_{i+1}$.} If $x_{i,j}$ is one of the two rightmost points on the top edge of $\g_i$, and $x_{i+1,j'}$ is one of the two leftmost points on the top edge of $\g_{i+1}$, then elementary calculations show that 
\[ \dist(\sigma_{i,j},\sigma_{i+1,j'}) \geq {2(5N)^{-1}} > (5\sqrt{2}N)^{-1}.\]
Otherwise,
\begin{align*} \dist(\sigma_{i,j},\sigma_{i+1,j'}) &\geq |x_{i,j}-x_{i+1,j'}| - 2(4\sqrt{2}(5N)^{-1})\\
&\geq \frac{25}{16} 8\sqrt{2}(5N)^{-1} - 2(4\sqrt{2}(5N)^{-1})\\ 
&> (5\sqrt{2}N)^{-1}.
\end{align*}
We may similarly treat the case that $x_{i,j}$ is on the bottom edge of $\g_i$ and $x_{i+1,j'}$ is on the bottom edge of $\g_{i+1}$.

\emph{Case 3.2: both $x_{i,j}$ and $x_{i+1,j'}$ are on the common edge of $\g_0^i$ and $\g^{i+1}_0$}. If $x_{i,j}=x_{i+1,j'}$, then trivially $\sigma_{i,j}\cap \sigma_{i+1,j'} \neq \emptyset$. If $|x_{i,j}-x_{i+1,j'}| = 1/N$, then $\sigma_{i,j}$ is parallel to $\sigma_{i+1,j'}$ and 
\[ \dist(\sigma_{i,j},\sigma_{i+1,j'}) = \tfrac5{16}8\sqrt{2}(5N)^{-1} = (\sqrt{2}N)^{-1} > (5\sqrt{2}N)^{-1}.\]
If $|x_{i,j}-x_{i+1,j'}| = 2/N$, then elementary calculations give
\[ \dist(\sigma_{i,j},\sigma_{i+1,j'}) = {2(5N)^{-1}} > (5\sqrt{2}N)^{-1}.\]
If $|x_{i,j}-x_{i+1,j'}| \geq 3/N$, then working as in Case 1.5, we get $\dist(\sigma_{i,j},\sigma_{i,j'}) > (5\sqrt{2}N)^{-1}$.

\emph{Case 3.3: $x_{i,j}$ is on the top edge of $\g_i$ and $x_{i+1,j'}$ is on the bottom edge of $\g_{i+1}$}. Then, $|x_{i,j}-x_{i+1,j'}| \geq 1$ and working as in Case 2, $\dist(\sigma_{i,j},\sigma_{i,j'}) > (5\sqrt{2}N)^{-1}$. We may similarly treat the case where $x_{i,j}$ is on the bottom edge of $\g_i$ and $x_{i+1,j'}$ is on the top edge of $\g_{i+1}$.

\emph{Case 3.4: $x_{i,j}$ is on the top edge of $\g_i$ and $x_{i+1,j'}$ is on the common edge of $\g_{i+1}, \g_{i+1}$}. Suppose first that $x_{i+1,j'}$ is one of the top 2 points on the left edge of $\g_{i+1}$ and $x_{i,j}$ is one of the three rightmost points of the top edge of $\g_i$. Then, elementary calculations show that
\[ \dist(\sigma_{i,j},\sigma_{i+1,j'}) \geq {(15-8\sqrt{2})(10\sqrt{2}N)^{-1} \geq (5\sqrt{2}N)^{-1}}.\]
Suppose now that neither $x_{i+1,j'}$ is one of the top 2 points on the left edge of $\g_{i+1}$, nor $x_{i,j}$ is one of the three rightmost points of the top edge of $\g_i$. Then, $|x_{i,j}-x_{i+1,j'}| \geq \sqrt{10}/N$ and working as in Case 1.5 we have $\dist(\sigma_{i,j},\sigma_{i,j'}) > (5\sqrt{2}N)^{-1}$.
\end{proof}

\begin{figure}[h]
\includegraphics[scale=0.8]{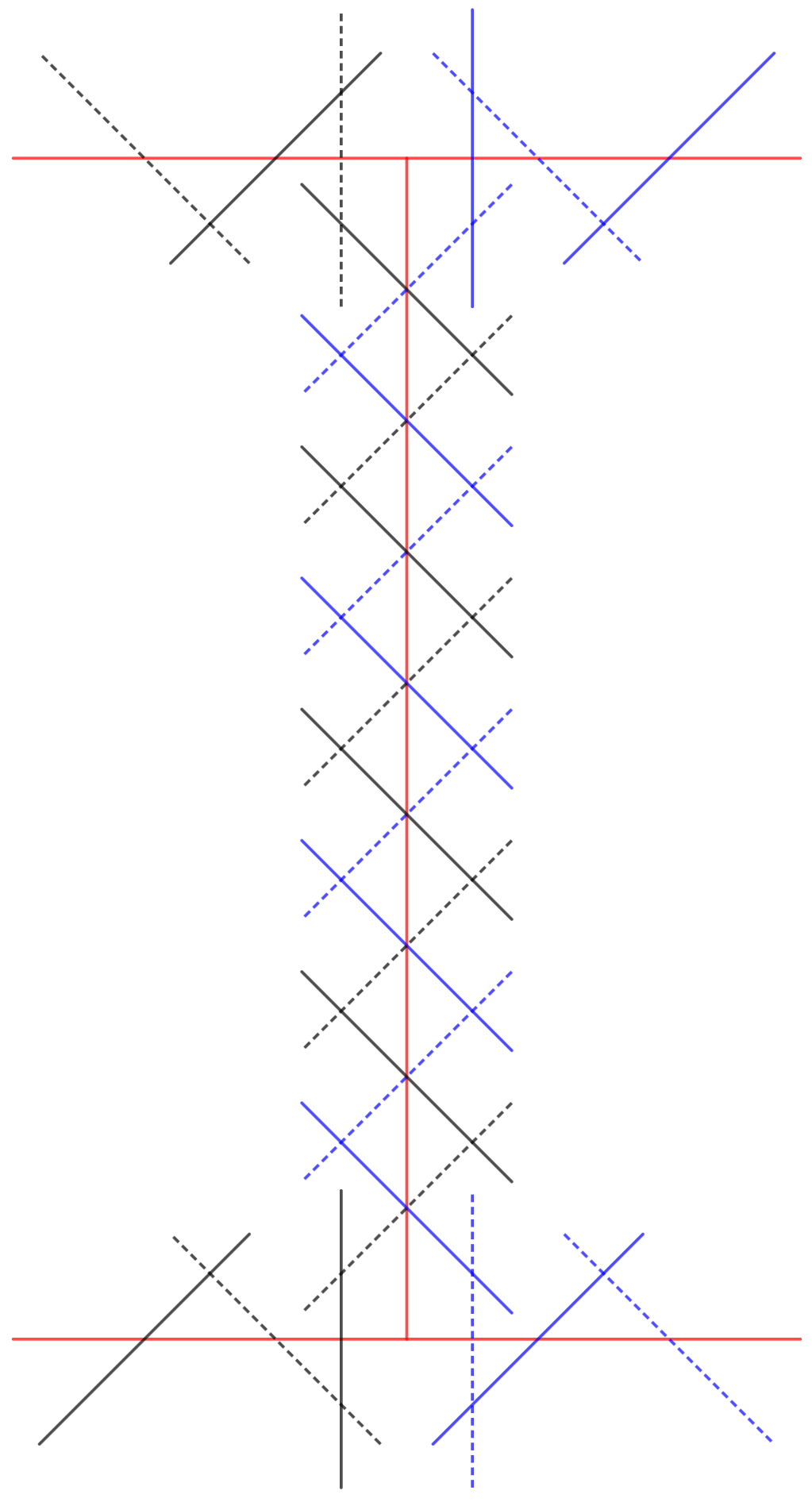}
\caption{The segments $\sigma_{i,j}$ (in black) and $\sigma_{i+1,j}$ (in blue) on a shared edge (in red) of $\gamma_i, \gamma_{i+1}$. As with Figure \ref{fig2}, for simplicity, we assume $N=9$.}
\label{fig3}
\end{figure}

Now for each $j\in \{1,\dots,2(c_i+1)N\}$ we define a copy $\tau_{i,j}$ of $\gamma$ in $\R^3$, scaled down by a factor of $8\sqrt{2}(5N)^{-1}$ with the following rules:
\begin{enumerate}
\item The projection of $\tau_{i,j}$ onto $\R^2\times\{0\}$ is the segment $\sigma_{i,j}$.
\item If $j$ is odd, then the projection of $\tau_{i,j}$ onto the $z$-axis is the segment 
\[ \{(0,0)\}\times [(-4C_g-1)\tfrac{\sqrt{2}}{5N}, (4C_g-1)\tfrac{\sqrt{2}}{5N}]. \]
\item If $j$ is even, then the projection of $\tau_{i,j}$ onto the $z$-axis is the segment 
\[ \{(0,0)\}\times [(-4C_g+1)\tfrac{\sqrt{2}}{5N}, (4C_g+1)\tfrac{\sqrt{2}}{5N}].\]
\end{enumerate}

For each $i \in \{1,\dots,g\}$ and $j \in \{1,\dots, 2(c_i+1)N\}$, we let $\psi_{i,j}:\R^3 \to \R^3$ be a similarity with scaling factor $8\sqrt{2}(5N)^{-1}$ such that $\psi_{i,j}(\gamma) = \tau_{i,j}$, and the image of the left edge of $\gamma$ is mapped to the edge of $\tau_{i,j}$, parallel to the $xy$-plane and with the highest third coordinate.

\begin{figure}[h]
\centering \includegraphics[scale=0.5]{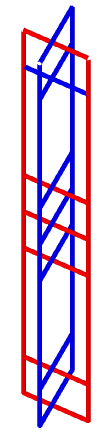}
\caption{The two curves $\tau_{i,j}$ (blue) and $\tau_{i',j'}$ (red) linked. Here we have chosen $g=6$.}
\label{fig4}
\end{figure}

\begin{lemma}\label{lem:ladders}
Let $i,i' \in \{1,\dots,g\}$, $j \in \{1,\dots, 2(c_i+1)N\}$, and $j'\in\{1,\dots,2(c_{i'}+1)N\}$ with $(i,j) \neq (i',j')$. 
\begin{enumerate}
\item We have that $\dist(\tau_{i,j},\tau_{i',j'}) \geq {(5\sqrt{2}N)^{-1}}$.
\item For all $x\in \tau_{i,j}$ we have that $\dist(x,\gamma_i) \leq (8C_g+6)\sqrt{2}(5N)^{-1}$.
\item We have that $\tau_{i,j}$ is linked with $\tau_{i',j'}$ if and only if $\sigma_{i,j}\cap\sigma_{i',j'}\neq \emptyset$. In the case they are linked, we have that $\psi_{i,j}(\gamma_k)$ is linked with $\psi_{i',j'}(\gamma_k)$.
\end{enumerate} 
\end{lemma}

\begin{proof}
For the first claim, if $\sigma_{i,j}\cap\sigma_{i',j'}=\emptyset$, then by Lemma \ref{lem:sigma}, (and given that $\sigma_{i,j},\sigma_{i',j'}$ are the projections of $\tau_{i,j},\tau_{i',j'}$ onto $\R^2\times\{0\}$, respectively) we have that 
\[ \dist(\tau_{i,j},\tau_{i',j'}) \geq \dist(\sigma_{i,j},\sigma_{i',j'}) \geq {(5\sqrt{2}N)^{-1}}.\]
If $\sigma_{i,j}\cap\sigma_{i',j'}\neq\emptyset$, then by Lemma \ref{lem:sigma} there are two possible cases. In either of these two cases, elementary calculations show that 
\[ \dist(\tau_{i,j},\tau_{i',j'}) \geq {(5\sqrt{2}N)^{-1}}.\]

For the second claim, we have that for all $x\in \tau_{i,j}$
\begin{align*} 
\dist(x,\gamma_i) \leq |x_{i,j}-x| &\leq \tfrac12\diam{\sigma_{i,j}} + 8C_g\sqrt{2}(5N)^{-1} +2\sqrt{2}(5N)^{-1}\\
&= (8{\color{red}C_g}+6)\sqrt{2}(5N)^{-1}.
\end{align*}

For the third claim, it is easy to see that if $\sigma_{i,j}\cap\sigma_{i',j'}\neq \emptyset$, then {$j$ and $j'$ have different parity},  $\sigma_{i,j}$ meets $\sigma_{i',j'}$ {traversally}, and {the the offset in the third coordinate of $\psi_{i,j}(\gamma_k)$ compared to $\psi_{i',j'}(\gamma_k)$ is a quarter of the length of the shortest side of $\psi_{i,j}(\gamma_k)$}. Therefore, $\psi_{i,j}(\gamma_k)$ is linked with $\psi_{i',j'}(\gamma_k)$. Assume now that $\sigma_{i,j}\cap\sigma_{i',j'}= \emptyset$. Then, $\tau_{i,j} \subset \sigma_{i,j}\times \R$, $\tau_{i',j'} \subset \sigma_{i',j'}\times \R$ with the two infinite strips having positive distance. Therefore, $\tau_{i,j}$ is not linked with $\tau_{i',j'}$.
\end{proof}

\subsection{A Cantor set}\label{sec:Cantor}
{Given $x\in\R^3$ and $r>0$, define $C(x,r) = x+ [-1,1]^3$. That is, $C(x,r)$ is the cube centered at $x$, of side-length $2r$, and with edges parallel to the axes.
Define the solid $g$-torus
\[ T^g = \bigcup_{p\in \gamma} C(p,(48\sqrt{3})^{-1}).
\]
Note that $\partial T^g$ is a compact PL 2-manifold.}

By the choice of $N$ and Lemma \ref{lem:ladders}(2), we have $ \psi_{i,j}\left(T^g\right) \subset T^g$ for all $i\in \{1,\dots,g\}$ and all $j \in \{1,\dots, 2(c_i+1)N\}$. By Lemma \ref{lem:ladders}(1), we have $\psi_{i,j}\left(T^g\right) \cap  \psi_{i,j}\left(T^g\right) = \emptyset$ for all $i,i' \in \{1,\dots,g\}$, all $j \in \{1,\dots, 2(c_i+1)N\}$, and $j'\in\{1,\dots,2(c_{i'}+1)N\}$ with $(i,j)\neq (i',j')$.

Let 
\begin{equation}\label{eq:phim} 
\{\phi_{1},\dots,\phi_m\} = \{\psi_{i,j} : i=1,\dots,g,\quad j=1,\dots,2(c_i+1)N\}
\end{equation}
with $m=2(C_g+g)N$. Define now the Cantor set
\begin{equation}\label{eq:Cantor} 
X_g = \bigcap_{n=1}^{\infty} \bigcup_{j_1,\dots, j_n \in \{1,\dots,m\}} \phi_{j_1}\circ\cdots\circ\phi_{j_n}(T^g).
\end{equation}

\section{Proving the genus is $g$}\label{sec:genus}

The goal of this section is to prove the following proposition.

\begin{proposition}\label{prop:genus}
The genus of $X_g$ is equal to $g$. Moreover, for each $x\in X_g$, the local genus $g_x(X_g)$ is $g$.
\end{proposition}

We start by establishing some terminology.

\begin{definition}

Let $g\geq 1$. We say that a solid genus $g$ torus $T$ embedded in $\R^3$ is {\it unknotted} if there is a homeomorphism $h:\R^3 \to \R^3$ such that $h(T) = T_0$, where $T_0$ is a solid torus with core curve
\begin{equation}
\label{eq:unknotted} 
\Sigma^{T_0} = \left ( [0,g] \times \{ 0,1\} \times  \{ 0 \} \right ) \cup \left ( \{0,1,\ldots, g \} \times [0,1] \times\{ 0 \} \right ).
\end{equation}
For $m\in \{1,\ldots, g\}$, let $\gamma^{T_0}_m$ be the loop in the core curve of $T_0$ given by
\[\gamma^{T_0}_m =  \left ( [m-1,m] \times \{ 0,1 \} \times  \{ 0 \}\right ) \cup \left ( \{m-1,m \} \times [0,1] \times \{ 0 \} \right ) ,\]
and set
\[ \gamma^T_m = h^{-1} ( \gamma^{T_0}_m ).\]
Then $\Sigma^T = \bigcup_{m=1}^g \gamma^T_m$ is a core curve for $T$.
\end{definition}

See Figure \ref{fig1} for an example of the core curve of an unknotted solid torus. 

\begin{definition}
Suppose that $T$ and $S$ are disjoint unknotted solid genus $g$ tori in $\R^3$. We say that $T$ and $S$ are {\it completely linked} if there is a homeomorphism $h:\R^3 \to \R^3$ such that $h(T) = T_0$ where $T_0$ has core curve given by \eqref{eq:unknotted} and $S_0$ has core curve given by
\begin{align*} 
\Sigma^{S_0} &= \left ( [1/2 , g + 1/2 ] \times  \{ 0 \} \times  \{ -1/2, 1/2 \} \right ) \\
 & \quad \cup \left ( \{1/2 , 3/2, \ldots, g+1/2 \} \times  \{ 0 \} \times  [-1/2 , 1/2 ] \right ).
\end{align*}
\end{definition}

We may enumerate the loops in the core curve for $S_0$ by $\gamma^{S_0}_m$, for $m=1,\ldots, g$ in the obvious way and may then consider the loops $\gamma^S_m = h^{-1} ( \gamma^{S_0}_m )$. Observe that $\gamma^T_1$ forms a Hopf link with $\gamma^S_1$ and is unlinked with $\gamma^S_m$ for $m=2,\ldots, g$. For $m \in \{2,\ldots, g\}$, $\gamma^T_m$ forms a Hopf link with both $\gamma^S_{m-1}$ and $\gamma^S_m$, and is unlinked with the other loops in the core curve for $S$.

See Figure \ref{fig4} for an example of completely linked solid genus $g$ tori.

\begin{lemma}
\label{lem:complink}
Let $g\geq 1$ and let $T$ and $S$ be disjoint completely linked solid genus $g$ tori with core curves $\Sigma^T,\Sigma^S$ respectively. If $U$ is a solid genus $g-1$ torus with $T \subset U$ and $\Sigma^S \cap \partial U = \emptyset$, then $\Sigma^S\subset U$.
\end{lemma}

\begin{proof}
Suppose for a contradiction that $\Sigma^S\subset \R^3 \setminus U$. As $\pi_1(U)$ is the free group on $g-1$ generators, this group has rank $g-1$.

Let $\Gamma_1$ be the subgroup of $\pi_1(U)$ generated by $[\gamma^T_1]$, the equivalence class of $\gamma^T_1$. As $\gamma^T_1 \subset T$ is linked with $\gamma^S_1 \subset S^3 \setminus  U$, it follows that $[\gamma^T_1]$ is non-trivial in $\pi_1(U)$. Hence $\Gamma_1$ has rank $1$.

For $k\in \{1,\ldots, g-1\}$, suppose that the subgroup $\Gamma_k$ of $\pi_1(U)$ that is generated by $[\gamma^T_1],\ldots, [\gamma^T_k]$ has rank $k$. 
Every element of $\Gamma_k$ can be represented by a loop that does not link with $\gamma^S_{k+1}$. 
As $\gamma^T_{k+1}$ is linked with $\gamma^S_{k+1} \subset \R^3 \setminus U$, it follows that $[\gamma^T_{k+1}]$ is both non-trivial in $\pi_1(U)$ and not an element of $\Gamma_k$. It follows that $\Gamma_{k+1}$ has rank $k+1$.

By induction, we conclude that $\Gamma_g$ is a rank $g$ subgroup of $\pi_1(U)$, which contradicts the fact that $\pi_1(U)$ has rank $g-1$.
\end{proof}

Next, if $\gamma$ is an unknotted circle in $\R^3$, then it bounds a topological disk. Any such disk is called a {\it filling disk}. In the special case where $\gamma$ is a planar topological circle, there is a unique filling disk which lies in the same plane as $\gamma$. This is called the {\it canonical filling disk} for $\gamma$.

\begin{definition}
If $T$ is a solid unknotted genus $g$ torus with core curve $\Sigma^T$, we define a {\it nice} collection of filling disks for $\Sigma^T$ to have the property that for each $m$ the filling disk for $\gamma^T_m$ arises as the image of a canonical filling disk for $\gamma^{T_0}_m$ under $h^{-1}$.
\end{definition}

The important point is that a nice collection of filling disks for $\Sigma^T$ consists of a collection of pairwise disjoint filling disks and so that if $\sigma_1,\sigma_2$ are any two closed loops contained in different filling disks, then $\sigma_1,\sigma_2$ are unlinked.

\begin{lemma}
\label{lem:fillingdisk}
Let $g\geq 1$ and let $T$ and $S$ be disjoint completely linked genus $g$ tori with core curves $\Sigma^T,\Sigma^S$ respectively. Let $D^S_m$, for $m\in \{1,\ldots, g\}$, be a nice collection of filling disks for $\Sigma^S$. Let $U\subset \R^3$ be a solid genus $g-1$ torus.
If $\Sigma^T\cup \Sigma^S \subset U$, then there exists a path in $U$ joining $\Sigma^T$ and $\Sigma^S$ contained in one of the filling disks $D^S_m$ for some $m\in \{ 1,\ldots, g\}$.
\end{lemma}

\begin{proof}
Suppose that $\Sigma^T\cup \Sigma^S \subset U$. For a contradiction, suppose that for each $m\in \{1,\ldots, g\}$, there is a loop $\sigma_m \subset D^S_m \cap (\R^3 \setminus U)$ which separates $\gamma^S_m$ from $D^S_m \cap \gamma^T_m$ in $D^S_m$. As $D^S_m$ is a nice collection of filling disks, these loops $\sigma_m$ are pairwise unlinked.

As in the proof of Lemma \ref{lem:complink}, we consider the subgroups $\Gamma_m$ of $\pi_1(U)$ generated by $[\gamma^T_1],\ldots, [\gamma^T_m]$. As $\gamma^T_1$ is linked with $\sigma_1 \subset \R^3 \setminus U$, it follows that $\Gamma_1$ has rank $1$. The same inductive argument as above, with $\sigma_m$ in place of $\gamma^S_m$, shows that $\Gamma_m$ has rank $m$ for $m\in \{1,\ldots, g \}$. This again contradicts the fact that $\pi_1(U)$ has rank $g-1$ and proves (i).

\end{proof}
Recall the construction of the Cantor set $X_g$ from \eqref{eq:Cantor}. Evidently $\operatorname{diam} X_g >0$. The construction of $X_g$ yields a defining sequence given by, for $n\geq 1$, 
\[ M_n = \bigcup_{ j_1,\ldots, j_n \in \{ 1,\ldots,m\} } \phi_{j_1} \circ \cdots \circ \phi_{j_n} (T^g) .\]

We fix some notation.
Given an integer $k\geq 0$, we denote by $\{1,\ldots, m\}^k$ the set of words formed from the alphabet $\{1,\ldots, m\}$ that have length exactly $k$. Conventionally, we set $\{1,\ldots, m\}^0 = \{\varepsilon\}$ where $\varepsilon$ is the empty word. We also denote by $\{1,\ldots, m\}^* = \bigcup_{k\geq 0}\{1,\ldots, m\}^k$ the set of all finite words formed from $\{1,\ldots, m\}$. For $k\geq 1$, denote by $\kappa : \{1,\ldots, m\}^k \to \{1,\ldots, m \}^{k-1}$ the forgetful map defined by
\[ \kappa(j_1 \cdots j_k) = j_1\ldots j_{k-1} \]
with $\kappa (j) = \varepsilon$ for $j\in \{1,\ldots, m\}$.
We have the alternative description of the defining sequence
\[ M_n = \bigcup_{w \in \{1,\ldots, m \}^n } \phi_w(T^g).\]
Here, if $w = j_1\ldots j_n$ then $\phi_w = \phi_{j_1} \circ \cdots \circ \phi_{j_n}$.

We are now in a position to prove that the genus of $X_g$ is $g$.

\begin{proof}[{Proof of Proposition \ref{prop:genus}}]
The construction of $X_g$ shows that its genus is at most $g$.

Suppose for a contradiction that the genus of $X_g$ is strictly smaller than $g$. Then we can find an alternative defining sequence $\widetilde{M}_n$ which contains solid genus $g-1$ tori of arbitrarily small diameter. Choose a solid genus $g-1$ torus $U$ that is a component of $\widetilde{M}$ of diameter at most $\alpha^{2}\operatorname{diam}(X^g) $, where $\alpha = 8\sqrt{2}(5N)^{-1}$ is the scaling factor of each $\phi_j$. 

As $U$ is a solid torus with $\partial U \subset \R^3 \setminus X_g$, we have $\dist ( \partial U, X_g) = \delta>0$. As 
\[ \sup_{x\in M_n} \dist (x,X_g) \to 0 \]
as $n\to \infty$, it follows that we may choose $p\in \N$ so that $\partial U \cap M_p  = \emptyset$. It follows that every component of $M_p$ is either contained in $U$ or contained in $\R^3 \setminus U$, and at least one is contained in $U$. 

{\bf Claim:} Suppose that $w\in \{1,\ldots, m\} ^ j$ and a core curve $\Sigma^w$ of $\phi_w(T^g)$ is contained in $U$ with a collection of nice filling disks $D^w_k$, for $k=\{1,\ldots, m\}$, contained in $\phi_{\kappa(w)}(T^g)$. Then there is a core curve $\Sigma^{\kappa(w)}$ of $\phi_{\kappa(w)}(T^g)$ that is contained in core curves of $\phi_{\kappa(w) k}(T^g)$, for $k\in \{1,\ldots, m\}$ and the collection of nice filling disks for each of these core curves.

To prove the claim, the collection of tori given by $\phi_{\kappa(w)k}(T^g)$, for $k\in \{1,\ldots, m\}$, form a chain of linked tori contained in the larger torus $\phi_{\kappa(w)} (T^g)$. Suppose that $\Sigma^{\kappa(w)k}$ is a core curve for each torus in the chain. As $\Sigma^w$ is contained in $U$, a repeated application of Lemma \ref{lem:complink} shows that all of the core curves $\Sigma^{\kappa(w)k}$ are contained in $U$.

Now, each core curve has a collection of nice filling disks that are certainly contained in $\phi_{\kappa(w)}(T^g)$. A repeated application of Lemma \ref{lem:fillingdisk} shows that we may construct a core curve of $\phi_{\kappa(w)}(T^g)$ with the claimed properties.

Returning the proof of the proposition, we have $\phi_w(T^g) \subset U$ for some $w\in\{1,\ldots, m \}^p$. Therefore a core curve $\Sigma^w$ of $\phi_w(T^g)$ satisfies the hypotheses of the claim, which yields a core curve $\Sigma^{\kappa(w)}$ for $\phi_{\kappa(w)}(T^g)$. We inductively find core curves $\Sigma^{\kappa^j(w)}$ for $\phi_{\kappa^j(w)}(T^g)$ contained in $U$ for $j=1,\ldots, p$.

In particular, choosing $j=p$ we obtain a core curve for $T^g$ that is contained in $U$.
This forces $\operatorname{diam} U \geq \diam (X^g)$ which contradicts $\operatorname{diam} (U) < \alpha^2 \operatorname{diam} (X^g)$.

The argument above shows that we cannot insert a genus $g-1$ handlebody into a defining sequence for $X_g$ in any non-trivial way. From this we conclude that the local genus of $X_g$ is $g$ at every $x\in X_g$.
\end{proof}

\section{A Julia set of genus $g$}\label{sec:julia}

The goal of this section is to construct a UQR map of $\overline{\R^3}$ that has $X_g$ as its Julia set. This along with Proposition \ref{prop:genus}, completes the proof of Theorem \ref{thm:1}.

\subsection{A basic covering map}

For each $n\in\N$ we denote by $\Sigma^n$ the associated planar curve $\g$ from Section \ref{sec:ladder} for the genus $n$, and by $\Gamma_{n,i}$ the simple closed curves $\g_i$ defined in \eqref{eq:gamma}. Fix for the rest of this section an integer $g\in\N$ and let 
$N_g$ be the integer defined in \eqref{eq:N}.

For each $n\in \{1,\dots,g\}$, each $i\in\{1,\dots,n\}$ and each $j\in\{1,\dots,(2c_{n,i}+2)N_g\}$ denote by $\alpha_{n,i,j}$ the planar segments $\sigma_{i,j}$ of length $8\sqrt{2}(5N_g)^{-1}$ intersecting $\Gamma_{n,i}$ as defined in Section \ref{sec:second}. Define also $\Sigma^g_{n,i,j}$ to be the curves $\tau_{i,j}$ defined just before Lemma \ref{lem:ladders} which are copies of $\Sigma^g$ scaled down by a factor of $8\sqrt{2}(5N_g)^{-1}$ and their projection on $\R^2\times\{0\}$ are the segments $\alpha_{n,i,j}$.

Define the metric $\delta$ on $\R^3$ given by
\begin{equation}\label{eq:delta} 
\delta((x_1,y_1,z_1),(x_2,y_2,z_2)) = \sqrt{\max\{|x_1-x_2|^2, |y_1-y_2|^2\}) + |z_1-z_2|^2}.
\end{equation}
Define for each $n\in\{1,\dots,g\}$, 
\[ T^n = \{p\in \R^3 : \delta(p,\Sigma^n) \leq \tfrac1{24}\}\]
and for each $i\in\{1,\dots,n\}$ and each $j\in\{1,\dots,(2c_{n,i}+2)N_g\}$,
\[ T^g_{n,i,j} = \{p\in \R^3 : \delta(p,\Sigma^g_{n,i,j}) \leq \sqrt{2}(15N_g)^{-1} \}.\]

The goal of this section is to construct the following map.

\begin{proposition}\label{prop:1}
There exists a degree $2^{\lceil \log_2{g} \rceil + 2} N_g$ BLD branched covering
\[ F: T^g \setminus \bigcup_{i,j} \int(T^g_{g,i,j}) \to \overline{B}(0,{4}C_{g,g}) \setminus \int(T^g) \]
such that for each $i,j$, {$F|\partial T^g_{g,i,j}$ is a similarity and} $F(\partial T^g_{g,i,j}) = \partial T^g$
\end{proposition}

The first step is given in the following lemma. 

\begin{lemma}\label{lem:folding}
For each $n\in\{2,\dots,g\}$ there exists a degree 2 BLD branched covering 
\[ F_n :  T^n \setminus \bigcup_{i,j}\int( T^g_{n,i,j} ) \to T^{\lceil \frac{n}2 \rceil} \setminus\bigcup_{i,j} \int( T^{g}_{\lceil \frac{n}2 \rceil,i,j} ) \]
such that for each $i,j$ there exist $i',j'$ with $F_n(\partial T^g_{n,i,j}) = \partial T^{g}_{\lceil \frac{n}2 \rceil,i',j'}$.
\end{lemma}

\begin{proof}[{Proof of Lemma \ref{lem:folding}}]
The construction of $F_n$ is different for the cases that $n$ is even or odd.

Assume first that $n=2k$ for some $k\in\N$.  Let 
\[ \iota_1:\R^3 \to \R^3 \qquad\text{with}\quad \iota_1(x,y,z) = (2C_{2k,k}-x,y,-z)\]
be the $\pi$-{radians} rotation with respect to the line $\ell_1 = \{z=0\}\cap\{x=C_{2k,k}\}$. Given such a rotation, we may realize the quotient $\R^3 / \langle \iota_1 \rangle$ concretely by a degree $2$ winding map $q_{\iota_1}:\R^3 \to \R^3$ that fixes pointwise the axis fixed by $\iota_1$. In cylindrical coordinates aligned with the fixed axis, $q_{\iota_1}$ is given by
\[ q_{\iota_1}(r,\theta,x_3) = (r,2\theta,x_3)\]
and it is well-known that $q_{\iota_1}$ is quasiregular, see for example \cite[p.13]{Rickman}. We have $q_{\iota_1}(\iota_1(x)) = q_{\iota_1}(x)$ for all $x\in \R^3$.

By the construction of the sequence $(a_{n,i})_{i=1}^n$ and the construction of sets $(T^g_{2k,i,j})_{i,j}$ we have that sets $T^{2k}$ and $\bigcup_{i,j}T^g_{2k,i,j}$ are both invariant under $\iota_1$ since both sets are symmetric with respect to the line $\ell_1$. The winding map $q_{\iota_1}$ then satisfies:
\begin{enumerate}
\item for each $i,j$ there exists $j'$ such that $q_{\iota_1}(T^g_{2k,i,j}) = q_{\iota_1}(T^g_{2k,2k-i+1,j'})$;
\item for each $i,j$, the image $q_{\iota_1}(T^g_{2k,i,j})$ is the image of $T^g_{2k,i,j}$ under a bi-Lipschitz homeomorphism of $\R^3$.
\end{enumerate} 

To obtain a BLD branched covering, we consider a BLD version of $q_{\iota_1}$ that we call $q$. Give $T^{2k}$ a $C^1$-triangulation $\zeta: |U| \to T^{2k}$ in the sense of \cite[p. 81]{Munkres} by a simplicial complex $U$ in $\R^3$ that respects the involution $\iota_1|T^{2k}$ and of which $\zeta^{-1}(\bigcup_{i,j}(T^g_{2k,i,j}))$ is a subcomplex. Identify $q_{\iota_1}(T^{2k})$ with a simplicial complex $V$ via $\xi:|V| \to q(T^{2k})$ in $\R^3$ of which $q_{\iota_1}(\bigcup_{i,j}(T^g_{2k,i,j}))$ is a subcomplex. This induces a simplicial map $\widetilde{q} :|U|\to |V|$ which is thus PL. We replace $q_{\iota_1}$ by $q:= \xi \circ \widetilde{q}\circ \zeta^{-1}$. As $\zeta$ and $\xi$ are both $C^1$ on the faces of the compact simplicial complexes $U$ and $V$ respectively, as and $\widetilde{q}$ is PL, it follows that $q$ is BLD.

Suppose now that $n=2k+1$. Recall that $c_{2k+1,k+1}=3$, $c_{k+1,k+1}=1$. We decompose $T^{2k+1} \setminus \bigcup_{i,j} \int(T^{g}_{2k+1,i,j})$ into six pieces and we decompose $T^{k+1} \setminus\bigcup_{i,j}\int( T^{g}_{k+1,i,j}) $ into three pieces as follows; see Figure \ref{fig5}. Let
\begin{align*} 
U_1 &= (T^{2k+1} \setminus \bigcup_{i,j} \int(T^{g}_{2k+1,i,j}) ) \cap\{x \leq C_{2k+1,k} + \tfrac12\}\\
U_2 &= (T^{2k+1} \setminus \bigcup_{i,j} \int(T^{g}_{2k+1,i,j}) ) \cap\{C_{2k+1,k} + {\tfrac12} \leq x \leq C_{2k+1,k} + \tfrac32\}\cap \{y\geq \tfrac12\}\\
U_3 &= (T^{2k+1} \setminus \bigcup_{i,j} \int(T^{g}_{2k+1,i,j}) ) \cap\{C_{2k+1,k} + \tfrac32 \leq x \leq C_{2k+1,k} + {\tfrac52}\}\cap \{y\geq \tfrac12\}\\
U_4 &= (T^{2k+1} \setminus \bigcup_{i,j} \int(T^{g}_{2k+1,i,j}) ) \cap\{C_{2k+1,k} + {\tfrac12} \leq x \leq C_{2k+1,k} + \tfrac32\}\cap \{y\leq \tfrac12\}\\
U_5 &= (T^{2k+1} \setminus \bigcup_{i,j} \int(T^{g}_{2k+1,i,j}) ) \cap\{C_{2k+1,k} + \tfrac32 \leq x \leq C_{2k+1,k} + {\tfrac52}\}\cap \{y\leq \tfrac12\}\\
U_6 &= (T^{2k+1} \setminus \bigcup_{i,j} \int(T^{g}_{2k+1,i,j}) ) \cap\{x \geq C_{2k+1,k} + {\tfrac52}\}.
\end{align*}
Let also
\begin{align*} 
V_1 &= ( T^{k+1} \setminus \bigcup_{i,j} T^{g}_{k+1,i,j} )  \cap\{x \leq C_{k+1,k} + \tfrac12\}\\
V_2 &= ( T^{k+1} \setminus \bigcup_{i,j} T^{g}_{k+1,i,j} )  \cap\{x \geq C_{k+1,k} + \tfrac12\}\cap\{y\geq \tfrac12\}\\
V_3 &= ( T^{k+1} \setminus \bigcup_{i,j} T^{g}_{k+1,i,j} )  \cap\{x \geq C_{k+1,k} + \tfrac12\}\cap\{y\leq \tfrac12\}.
\end{align*}

\begin{figure}
\includegraphics[width=0.85\textwidth]{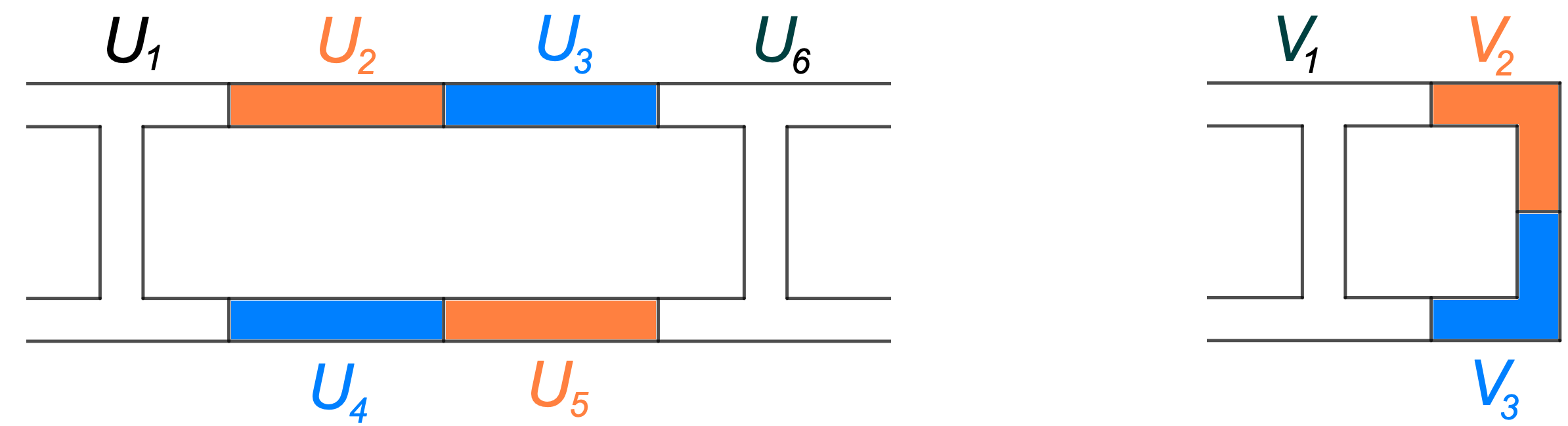}
\caption{The decompositions of  $T^{2k+1}$ (left) and of $T^{k+1}$ (right). For simplicity we have omitted the sets $T^{g}_{2k+1,i,j}$ and $T^{g}_{k+1,i,j}$.}
\label{fig5}
\end{figure}

Let  
\[ \iota_2 : \R^3 \to \R^3 \qquad\text{with}\quad \iota_2(x,y,z) = (2C_{2k+1,k}+{3}-x,1-y,z)\]
be the $\pi$-{radians} rotation with respect to the line $\{y=1/2\}\cap\{x=C_{2k+1,k}+\tfrac32\}$ and let
\[ \iota_3 : \R^3 \to \R^3 \qquad\text{with}\quad \iota_3(x,y,z) = (y+C_{k+1,k},C_{k+1,k+1}-x,z)\]
be the $\pi/2$-radians rotation with respect to the line $\{y=1/2\}\cap\{x=C_{k+1,k}+\tfrac12\}$.

Define $F_{2k+1}|U_1 = \text{Id}$ (which maps $U_1$ onto $V_1$) and $F_{2k+1}|U_6 = \iota_2$ (which maps $U_6$ onto $V_1$). We claim that there exists a bi-Lipschitz homeomorphism $h: U_2 \to V_2$ such that 
\begin{enumerate}
\item for each $j$, there exists unique $j'$ such that $h|\partial T^g_{2k+1,k+1,j}\cap U_2$ is an isometry mapping $\partial T^g_{2k+1,k+1,j}{\cap U_2}$ onto $\partial T^g_{k+1,k+1,j'}{\cap V_2}$;
\item $h|\partial U_1\cap\partial U_2 = \text{Id}$ and $h|\partial U_3\cap\partial U_2$ is a {clockwise $\pi/2$-radians rotation with respect to the line $\{(x,y,z) : x=C_{2k+1,k}+1, y=\frac12\}$ that maps $\partial U_3\cap\partial U_2$ onto $\partial V_3\cap\partial V_2$}.
\end{enumerate}
The construction of $h$ is elementary but tedious and we postpone its proof until Appendix \ref{ap:ext}. Assuming the existence of $h$, we define 
\begin{enumerate}
\item $F_{2k+1}|U_2 = h$ (which maps $U_2$ onto $V_3$), 
\item $F_{2k+1}|U_3 = \iota_3 \circ h(x-\tfrac34,y,z)$ (which maps $U_3$ onto $V_3$), 
\item $F_{2k+1}|U_4 = (F_{2k+1}|U_3)\circ\iota_2$ (which maps $U_4$ onto $V_3$), and 
\item $F_{2k+1}|U_5 = (F_{2k+1}|U_2)\circ\iota_2$ (which maps $U_5$ onto $V_2$). 
\end{enumerate}
It is easy to see that $F_{2k+1}$ is a degree 2 BLD branched covering.
\end{proof}

We are now ready to prove Proposition \ref{prop:1}. The proof follows the arguments in \cite[\textsection 4.1]{FS2} almost verbatim.

\begin{proof}[{Proof of Proposition \ref{prop:1}}]
Applying Lemma \ref{lem:folding} a total of $\lceil \log_2{g} \rceil$ many times, we obtain a degree $2^{\lceil \log_2{g} \rceil}$ BLD map 
\[ G : T^g \setminus \bigcup_{i,j} \int(T^g_{{g},i,j}) \to T^{1} \setminus\bigcup_{j=1}^{4N_g} \int(T^{g}_{1,1,j})\]
such that for each $i,j$ there exist $j'$ with $G(\partial T^g_{g,i,j}) = \partial T^{g}_{1,{1},j'}$. 

It remains to construct a degree $4N_g$ BLD map
\[ T^{1} \setminus\bigcup_{j=1}^{4N_g} \int(T^{g}_{1,1,j}) \to \overline{B}(0,2C_{g,g}) \setminus \int(T^g).\] 

We apply a bi-Lipschitz map $\Phi:T^1 \to {\R^3}$ that modifies $T^1$ in two ways. Firstly, we translate $T^1$ so that its core curve $\g_1$ is the 2-dimensional unit square
\[ \{(x,y,z) : z=0, \quad \max\{|x|,|y|\}=1\}.\]
Then, we apply a bi-Lipschitz map that is radial with respect to the $z$-axis so that 
\begin{enumerate}
\item $\Phi(T_1)$ is the closed $\frac12$-neighborhood of the circle $\S^1\times\{0\}$ in $\R^3$,
\item all the sets $\Phi(T^g_{1,1,1}),\dots, \Phi(T^g_{1,1,4N_g})$ satisfy
\[ \rho(\Phi(T^g_{1,1,j})) = \Phi(T^g_{1,1,j+2}) \quad \text{for $j\in\{1,\dots,4N_g\}$}\]
(with the convention $T^g_{1,1,4N_g+1}=T^g_{1,1,1}$ and $T^g_{1,1,4N_g+2} = T^g_{1,1,2}$) where $\rho$ is the rotation about the $z$-axis by an angle $\pi/N_g$,
\[ \rho(r,\theta,z) = (r,\theta+ \pi/N_g,z).\]
\end{enumerate}

This deformation is made to preserve the fact that all $\Phi(T^g_{1,1,j})$ are similar to each other. Finally, if necessary, rotate $\Phi(T^1)$ around the $z$-axis to ensure that the set $\bigcup_{j=1}^{4N_g} \int(\Phi(T^{g}_{1,1,j}))$ is symmetric with respect to a rotation about the $x$-axis by an angle $\pi$; see Figure \ref{fig6}.

\begin{figure}[h]
\includegraphics[width=0.3\textwidth]{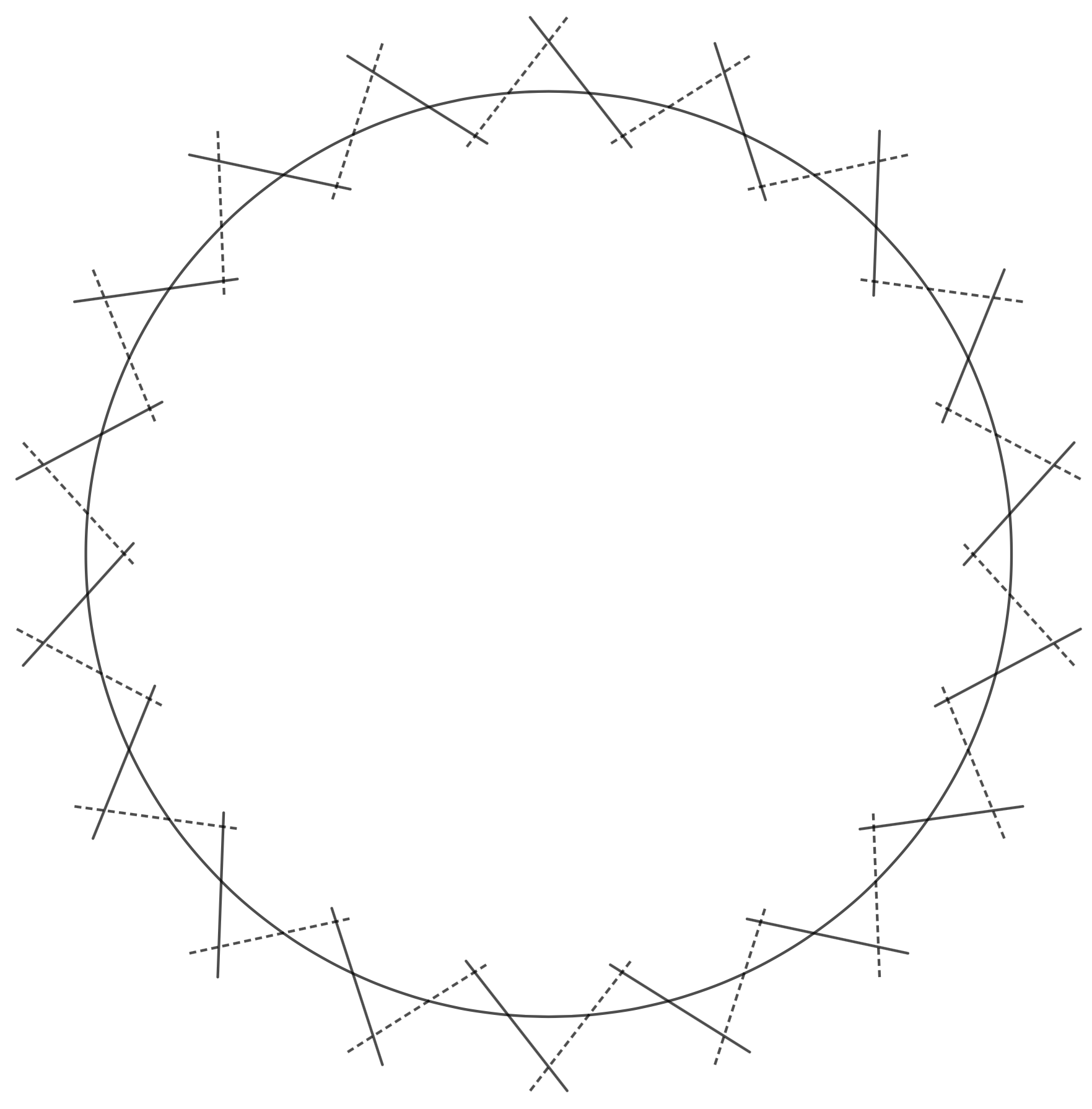}
\caption{The projections of ${\Phi(\g_1)}$ (the core curve of ${\Phi(T^1)}$) and of $({\Phi(T^{g}_{1,1,j})})_{j=1}^{4N}$ on the $xy$-plane. For simplicity we have chosen $N_g=9$.}
\label{fig6}
\end{figure}

Let $\omega : \R^3 \to \R^3$ be the degree $2N_g$ winding map
\[ \omega(r,\theta,z) = (r,2N_g\theta,z).\]
Then $\omega|{\Phi(T^1)} : {\Phi(T^1)} \to {\Phi(T^1)}$ is an unbranched covering that maps all {$\Phi(T^g_{1,1,j})$} with odd indices $j$ to $\omega({\Phi(T^g_{1,1,1})})$ and all ${\Phi(T^g_{1,1,j})}$ with even indices $j$ to $\omega({\Phi(T^g_{1,1,2})})$. By construction, $\omega({\Phi(T_{1,1,1}^g)})$ and $\omega({\Phi(T_{1,1,2}^g)})$ are linked inside ${\Phi(T^1)}$ (see Figure \ref{fig7}) and are symmetric to each other via a rotation about the $x$-axis by an angle $\pi$. Let $\iota$ be the involution for the latter rotation, that is
\[ \iota(x,y,z) = (x,-y,-z).\]
The associated winding map $q_{\iota}$ is a degree 2 sense preserving map under which $q_{\iota}(\omega({\Phi(T^g_{1,1,1})})) = q_{\iota}(\omega({\Phi(T^g_{1,1,2})}))$ is the image of ${\Phi(T^g)}$ under a bi-Lipschitz map of $\R^3$.

\begin{figure}[h]
\includegraphics[width=0.4\textwidth]{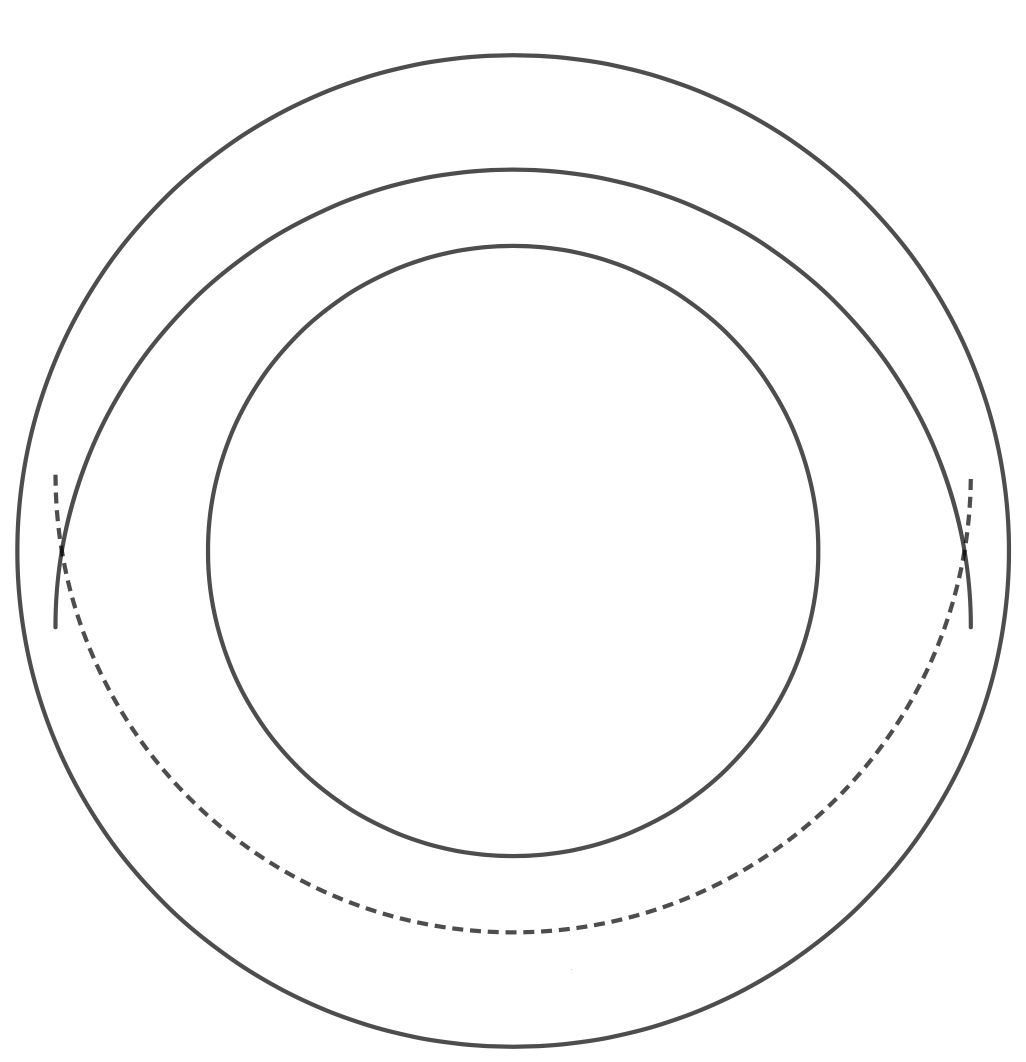}
\caption{{$\omega(\Phi(T^g_{1,1,1}))$ and $\omega(\Phi(T^g_{1,1,2}))$ linked inside $\omega(\Phi(T^1)) = \Phi(T^1)$.}}
\label{fig7}
\end{figure}

Post-composing with more bi-Lipschitz deformations, the map $q_{\iota}\circ\omega\circ G$ is a degree $4N_g 2^{\lceil \log_2{g} \rceil}$ branched covering from {$T^g\setminus \bigcup_{i,j} \text{int}(T_{g,i,j}^g)$} onto {$\overline{B}(0,4)\setminus \text{int}(T^g)$} mapping each $T^g_{g,i,j}$ onto $T^g$. Following the arguments in the proof of Lemma \ref{lem:folding} we can obtain a BLD version of the map $q_{\iota}\circ\omega\circ G$. This completes the proof of the proposition.
\end{proof}

{Note that since the branched covering in the previous Lemma is BLD, it is also quasiregular.}

\subsection{Construction of a UQR map} \label{subsec:uqr}
The construction of the UQR map of Theorem \ref{thm:1} follows closely the ideas in \cite[Section 5]{FW} and \cite[\textsection 4.2]{FS2} so we only sketch the arguments. We require the following two results.

\begin{theorem}
\label{thm:degreeUQR}
For every $d\in\N$ with $d>1$ and for every $n\in\N$, there is a UQR map $h: \overline{\R^3} \to \overline{\R^3}$ of degree $2^nd^2$ with Julia set $J(h) = \mathbb{S}^2$. {In addition}, for any $r>0$, $h(B(0,r)) = B(0,r^{2^{n/2}d})$.
\end{theorem}

We defer the proof of this result to Appendix \ref{ap:uqr}.
{We finally require a PL version of a result due to Berstein and Edmonds \cite{BE} on extending branched coverings over PL cobordisms. For the case with maximum generality in the PL setting, we refer to  \cite{PRW}. See also \cite[Theorem 0.3]{HR}.}

\begin{theorem}[{{See \cite[Theorem 6.2]{BE} and \cite[Theorem 3.1]{PRW}}}]\label{thm:BE}
Let $W$ be a connected, compact, oriented PL 3-manifold in some $\R^n$ whose boundary $\partial W$ consists of two components $M_0$ and $M_1$ with the induced orientation. Let $W' = N\setminus (\operatorname{int}(B_0) \cup \operatorname{int}(B_1))$ be an oriented PL 3-sphere $N$ in $\R^n$ with two disjoint polyhedral 3-balls removed, and have the induced orientation on its boundary. Suppose that $\phi_i : M_i \to \partial B_i$ is a sense-preserving oriented {PL} branched covering of degree $d\geq 3$ for each $i=0,1$. Then there exists a sense-preserving PL branched covering $\phi: W \to W'$ of degree $d$ that extends $\phi_0$ and $\phi_1$.
\end{theorem}

Recall the constant $C_{g,g}$ from \textsection\ref{sec:seq} and set
\[ R = (4C_{g,g})^{2\sqrt{N_g}2^{\frac12 \lceil \log_2{g}\rceil}}.\]
Let $B_0 = B(0,4C_{g,g})$, $B_{-1} = B(0,R)$. We decompose $\R^3$ in two different ways:
\begin{align*}
\R^3 &= \left(\bigcup_{i,j}T^g_{g,i,j}\right) \cup \left( T^g \setminus \bigcup_{i,j}T^g_{g,i,j} \right) \cup (B_0 \setminus T^g) \cup (\R^3 \setminus B_0)
\end{align*}
and
\begin{align*}
\R^3 &= T^g \cup (B_0 \setminus T^g) \cup (B_{-1} \setminus B_0) \cup (\R^3\setminus B_{-1}).
\end{align*}
Define now a map $f:\R^3 \to \R^3$ in the following way.
\begin{enumerate}
\item Set 
\[ f : T^g \setminus \bigcup_{i,j}\text{int}(T^g_{g,i,j}) \to B_0 \setminus \int(T^g)\]
to be the map of Proposition \ref{prop:1}.
\item Extend $f$ to $\bigcup_{i,j}T^g_{g,i,j}$ by setting 
\[ f|T^g_{g,i,j}  = \psi_{i,j}^{-1} : T^g_{g,i,j} \to T^g.\]
Recall the definitions of $\psi_{i,j}$ from \textsection\ref{sec:Cantor}.
\item Define $f:\R^3\setminus \text{int}(B_0) \to \R^3 \setminus \text{int}(B_{-1})$ to be the restriction of the UQR map of degree $2^{\lceil \log_2{g}\rceil}4N_g$ from Theorem \ref{thm:degreeUQR}. Note that $f|\R^3\setminus \text{int}(B_0)$ is orientation preserving and maps $S(0,4C_{g,g})$ onto $S(0,R)$.
\item Since $f|\partial B_0$ is a BLD degree $2^{\lceil \log_2{g}\rceil}4N_g$ map onto $\partial B_{-1}$, and since $f|\partial T^g$ is also a degree $2^{\lceil \log_2{g}\rceil}4N_g$ BLD map onto $\partial B_0$, by Theorem \ref{thm:BE} there exists a degree $2^{\lceil \log_2{g}\rceil}4N_g$ BLD extension $f: B_0 \setminus \text{int}(T^g) \to B_{-1}\setminus \text{int}(B_0)$. {It is understood here that $C_1$-triangulation has been carried out on $B_0 \setminus \text{int}(T^g)$ and $B_{-1}\setminus \text{int}(B_0)$ before applying Theorem \ref{thm:BE}.}
\end{enumerate}

The map $f:\R^3 \to \R^3$ defined above is UQR and of polynomial type \cite[Lemma 4.1]{FS2} (see also \cite[Lemma 5.1]{FW}) and its Julia set is equal to $X_g$ \cite[Lemma 4.2]{FS2} (see also \cite[Lemma 5.2]{FW}).

\section{Genus and Julia sets}\label{sec:genusjulia}

First in this section, we prove Theorem \ref{thm:3}. In fact we will show something stronger. Namely that if $X$ is a Cantor set and the Julia set of a hyperbolic UQR map, then there exists a defining sequence that, up to similarities, contains only finitely many different elements. 

Let us first recall two notions from metric geometry. A metric space $(X,d)$ is \emph{$c$-uniformly perfect} if there exists $c\geq 1$ such that for all $x\in X$ and all $r\in (0,\diam{X})$, $\overline{B}(x,r)\setminus B(x,r/c)$. A metric space $(X, d)$ is \emph{$c$-uniformly disconnected} \cite{DSbook} if there exists $c \geq 1$ such that for any $r \in (0,\diam{X})$ and any $x \in X$ there exists a set $E \subset X$ containing $x$ such that $\diam{E} \leq r$ and $\dist(E, X \setminus E) \geq r/c$. 

Before stating our key lemma, we introduce some terminology. Fix $N\in\N$. We denote by $\varepsilon$ the empty word. Given an integer $n\geq 0$ we denote by $\{1,\dots,N\}^n$ all words formed from the alphabet $\{1,\dots,N\}$ that have exactly $n$ letters with the convention $\{1,\dots,N\}^0 = \{\varepsilon\}$. Define also the set of all finite words
\[ \{1,\dots,N\}^* = \bigcup_{n\geq 0} \{1,\dots,N\}^n.\]
The length of a finite word $w$ is denoted by $|w|$. Two finite words $w,u$ can be concatenated to $wu$ in the obvious way.

The next lemma states that sets that are uniformly perfect and uniformly disconnected admit a defining sequence which, up to similarities, contains finitely many elements in a quantitative way. We prove in fact a stronger version which yields an interesting uniformization of such Cantor sets which may be on independent interest; see Appendix \ref{ap:QRunif}.

\begin{lemma}\label{lem:finite defining sequence}
Given $c>1$ there exist constants $p,N,C_0\in \N$, a finite collection $\{\mathcal{M}_1,\dots,\mathcal{M}_l\}$ of PL handlebodies in $\R^3$, and a finite collection $\{\mathcal{N}_1,\dots,\mathcal{N}_q\}$ of PL 3-manifolds with boundary, each $\mathcal{N}_j$ having at least three boundary components, that satisfy the following. If $X\subset \R^3$ is compact, $c$-uniformly perfect and $c$-uniformly disconnected, then there exist a ``dictionary'' $\mathcal{W} \subset \{1,\dots,N\}^*$, a map ${\bf i}:\mathcal{W} \to \{1,\dots,l\}$, a map ${\bf j}:\mathcal{W} \to \{1,\dots,q\}$, and similarities $\{\phi_w:\R^3 \to\R^3\}_{w\in \mathcal{W}}$ such that
\begin{enumerate}
\item[(P1)] 
\begin{enumerate}
\item the empty word $\varepsilon \in \mathcal{W}$,
\item for every $w\in \mathcal{W}$ there exists $N_w \in \{2,\dots,N\}$ such that $wi\in \mathcal{W}$ if and only if $i\in\{1,\dots,N_w\}$,
\item if $iw \in \mathcal{W}$ for some $w\in \{1,\dots,N\}^*$ and $i\in \{1,\dots,N\}$, then $w\in \mathcal{W}$,
\end{enumerate} 
\item[(P2)] for all $w\in\mathcal{W}$, $\phi_w$ has scaling factor $p^{-|w|}$,
\item[(P3)] for all $w\in\mathcal{W}$ and all $i\in\{1,\dots,N_w\}$, $\phi_{wi}(\mathcal{M}_{{\bf i}(wi)}) \subset \phi_w(\mathcal{M}_{{\bf i}(w)})$ and
\[ \dist(\phi_{wi}(\mathcal{M}_{{\bf i}(wi)}),\partial \phi_{w}(\mathcal{M}_{{\bf i}(w)})) \geq C_0^{-1}p^{-|w|-1},\]
\item[(P4)] for all distinct $w,w' \in \mathcal{W}$ with $|w|=|w'|$
\[ \dist(\phi_{w}(\mathcal{M}_{{\bf i}(w)}), \phi_{w'}(\mathcal{M}_{{\bf i}(w')})) \geq p^{-|w|},\]
\item[(P5)] for all $w\in\mathcal{W}$, $\phi_{w}(\mathcal{M}_{{\bf i}(w)})\cap X \neq \emptyset$ and
\[ \dist(\partial \phi_{w}(\mathcal{M}_{{\bf i}(w)}), X) \geq p^{-|w|}, \qquad p^{-|w|}\leq \diam{\phi_{w}(\mathcal{M}_{{\bf i}(w)})} \leq C_0 p^{-|w|},\]
\item[(P6)] for each $w\in \mathcal{W}$, 
\[ \mathcal{M}_{{\bf i}(w)} \setminus \bigcup_{i=1}^{N_w}\phi_w^{-1}\circ\phi_{wi}(\mathcal{M}_{{\bf i}(wi)}) = \mathcal{N}_{{\bf j}(w)},\]
\item[(P7)] the set $X$ is the limit of the $k$-th level approximations:
\[ X = \bigcap_{k\geq 0} \bigcup_{\substack{w \in \mathcal{W} \\ |w|=k}}\phi_{w}(\mathcal{M}_{{\bf i}(w)}).\]
\end{enumerate}
\end{lemma}

\begin{proof}
The proof uses MacManus’ cubical approximation of uniformly disconnected sets \cite{MM2}. Given $\d>0$, define $\mathcal{D}_{\d}$ to be the collection of connected $3$-manifolds with boundary that are the union of finitely many cubes in the collection
\[ \left\{ [m_1\d,(m_1+1)\d]\times[m_2\d,(m_2+1)\d]\times [m_3\d,(m_3+1)\d] : m_1,m_2,m_3\in \mathbb{Z} \right\}.\]

Let $X$ be $c$-uniformly perfect and $c$-uniformly disconnected. By \cite[Corollary 5.2]{BV}, there exist constants $C_0,p,N \in\N$ depending only on $c$, there exists a dictionary $\mathcal{W}\subset \{1,\dots,N\}^*$, and there exists a family $\{M_w:w\in\mathcal{W}\}$ of $3$-manifolds with boundary in $\R^3$ such that
\begin{enumerate}
\item $\mathcal{W}$ satisfies (P1),
\item for all $w\in \W$, $M_w \in \mathcal{D}_{p^{-|w|}}(X)$ and $\diam{M_w} \leq C_0p^{-|w|}$,
\item for all distinct $w,w' \in \W$ with $|w|=|w'|$,  $\dist(M_w,M_{w'}) \geq p^{-|w|}$,
\item for all $w\in\W$ and for all $i\in \{1,\dots,N_w\}$, $M_{wi} \subseteq M_w$ and
$$ \dist(M_{wi},\partial M_w) \geq C_0^{-1}p^{-|w|},$$
\item for all $w\in \W$, the intersection $X \cap M_w \neq \emptyset$ and $\dist(\partial M_w, X) \geq p^{-|w|}$,
\item the set $X$ is the limit of the $k$-th level approximations: \[ X = \bigcap_{k\geq 0} \bigcup_{\substack{w \in \mathcal{W} \\ |w|=k}}M_{w}.\]
\end{enumerate} 

Note that the manifolds $M_w$ are not assumed to have connected boundaries and that $N_w$ may equal 1 for some words $w\in \mathcal{W}$.

The first issue can be resolved as follows. For each $k\in\N$ and each $w\in \mathcal{W}$ with $|w|=k$ replace $M_w$ by $\overline{\R^3\setminus U_w}$ where $U_w$ is the unbounded connected component of $\R^3 \setminus M_w$. If for some $w,u \in \mathcal{W}$ with $|w|=|u|=k$ we have $M_u \subset M_w$, then we remove $u$ from $\mathcal{W}$. After re-indexing, we obtain a new dictionary $\mathcal{W}$  and a new collection $\{M_w: w\in \mathcal{W}\}$ of 3-manifolds with boundary that have connected boundaries and satisfy all properties above (with the same constants $p,C_0,N$). 

To fix the second issue, we remark that by $c$-uniform perfectness and by (4) above, we have that for all $w\in\mathcal{W}$,
\[ \diam(X \cap M_w) \geq c^{-1}p^{-|w|}.\]
Therefore, assuming that $p> C_0 c$, we have by (4) above that for all $w\in\mathcal{W}$ and for all $i\in\{1,\dots,N\}$ such that $wi\in \mathcal{W}$,
\[ \diam{M_{wi}} < \diam(X \cap M_w)\]
which yields that $N_w \geq 2$.

Denote by $\mathscr{M}$ the collection of all PL handlebodies $M\in \mathcal{D}_1$ such that $M\subset [0,C_0+1]^3$. Furthermore, denote by $\mathscr{N}$ the collection of all $M \in \mathcal{D}_1$ such that $M \subset[0,C_0+1]^3$ and $M$ has at least three boundary components. Since collections $\mathscr{M}$ and $\mathscr{N}$ are finite, we can enumerate them $\mathscr{M} = \{\mathcal{M}_1,\dots,\mathcal{M}_l\}$ and $\mathscr{N} = \{\mathcal{N}_1,\dots,\mathcal{N}_q\}$.

For each $w\in \mathcal{W}$, let $\phi_w$ be a similarity map of $\R^3$ with scaling factor $p^{-|w|}$ such that $\phi_w^{-1}(M_w)\in \mathcal{D}_1$. Note that $\phi_w^{-1}(M_w)$ has diameter at most $C_0$ so modifying $\phi_w$ we may further assume that $\phi_w^{-1}(M_w) \subset [0,\lceil C_0\rceil +1]^3$. Therefore, for each $w\in\mathcal{W}$, we have $\phi_w^{-1}(M_w) \in \mathscr{M}$. Similarly, for each $w\in\mathcal{W}$, we have $\phi_w^{-1}(M_w \setminus \bigcup_{i=1}^{N_w}M_{wi}) \in \mathscr{N}$. This completes the proof of the lemma.
\end{proof}

We can now show Theorem \ref{thm:3}.

\begin{proof}[{Proof of Theorem \ref{thm:3}}]
Let $X$ be the Julia set of a hyperbolic UQR map $f:\overline{\R^3} \to \overline{\R^3}$. Without loss of generality, we may assume that $\infty \not\in X$. By \cite[Theorem 1.1]{FV}, $X$ is compact and uniformly disconnected, and by \cite[Theorem 1.1]{FN11}, $X$ is uniformly perfect. 

Let $\mathcal{W}$, $\{\mathcal{M}_1,\dots,\mathcal{M}_l\}$, ${\bf i}:\mathcal{W} \to\{1,\dots,l\}$, and $\{\phi_w:\R^3 \to \R^3\}_{w\in \mathcal{W}}$ be the dictionary, finite collection of handlebodies, function, and similarities as in Lemma \ref{lem:finite defining sequence}. Then,
\[ \left(\bigcup_{w\in \mathcal{W}, |w|=i}\phi_w(\mathcal{M}_{{\bf i}(w)}) \right)_i  \]
is a defining sequence for $X$ and
\[ g(X) \leq \sup_{w\in \mathcal{W}}g(\phi_w(\mathcal{M}_{{\bf i}(w)})) \leq \max_{i=1,\dots,l}g(\mathcal{M}_i) < \infty. \qedhere\]
\end{proof}

Next, we prove Theorem \ref{thm:4}. The key is the following lemma on the local genus.

\begin{lemma}
\label{lem:localgenus}
Let $X$ and $Y$ be Cantor sets in $\R^3$ and let $x\in X$. Suppose there exists a neighborhood $U$ of $x$ and a homeomorphism $h$ from $U$ onto a neighborhood $V$ of $y=h(x)$ such that $h(X\cap U) = Y\cap V$. Then $g_x(X) = g_y(Y)$.
\end{lemma}

\begin{proof}
Let us first fix a defining sequence $(M_n)$ for $X$ and recall that $M_n^x$ is the component of $M_n$ containing $x$. Now consider any defining sequence $\widetilde{M}_n$ of $Y$. As $\diam \widetilde{M}_n^y \to 0$ as $n\to \infty$, there exists $N\in \N$ such that $\widetilde{M}_n^y \subset V$ for $n\geq N$.

We will build a new defining sequence $\widehat{M}_n$ for $Y$ as follows. If $n\geq N$, we leave all the components of $\widetilde{M}_n$ alone, except for $\widetilde{M}_n^y$. As $\widetilde{M}_n^y$ is a handlebody with $\partial \widetilde{M}_n^y \subset V \setminus Y$, it follows that $h^{-1} (  \widetilde{M}_n^y )$ is a handlebody in $U$ with boundary contained in $U\setminus X$. 

In particular, there exists a minimal integer $k=k(n)$ such that the union of the boundaries of components of $M_k$ does not intersect $h^{-1} (  \widetilde{M}_n^y )$. Let $\Omega_k$ denote the union of the components of $M_k$ contained in $h^{-1} (  \widetilde{M}_n^y )$. Then for our new defining sequence of $Y$, we may replace $\widetilde{M}_n^y$ with $h(\Omega_k)$.

In $\widehat{M}_n'$ we have $\widehat{M}_n^y = h ( M_k^x)$. As $h$ is a homeomorphism, the genus of $\widehat{M}_n^y$ is the same as $M_k^x$. Performing the same procedure for all defining sequences $M_n$ of $X$ and taking an infimum, we see that
\[ g_y(Y) \geq g_X(X) .\]
Switching the roles of $x\in X$ and $y\in Y$ and using the fact that $h^{-1}$ is also a homeomorphism, the above argument shows that
\[ g_x(X) \geq g_y(Y),\]
which completes the proof.
\end{proof}

\begin{proof}[Proof of Theorem \ref{thm:4}]
By Theorem \ref{thm:3}, the genus of $J(f)$ is finite. Since $g_x(J(f)) \leq g(J(f))$ for any $x\in J(f)$, the local genus is also finite at every point. So suppose $g_x(J(f)) = g$ and let $y$ be in the grand orbit of $x$. Then there exists an integer $n$ such that either $f^n(x) = y$ or $f^n(y) = x$.

As $f$ is hyperbolic, $f^n$ is a local homeomorphism at every point of $J(f)$. By complete invariance, $y\in J(f)$. Lemma \ref{lem:localgenus} now yields the result.
\end{proof}

\section{Non-constant local genus}\label{sec:genusg0}

In this section, we modify the construction from Section \ref{sec:julia} to give an example of a UQR map $f$ with $J(f)$ a genus $g$ Cantor set, for $g\geq 1$, and so that local genus of both $0$ and $g$ is achieved. 

\begin{proof}[Proof of Theorem \ref{thm:5}]
Recall the BLD map 
\[ F: T^g \setminus \bigcup_{i,j} \int(T^g_{g,i,j}) \to \overline{B}(0,{4}C_{g,g}) \setminus \int(T^g)\] 
from Proposition \ref{prop:1}. 
For brevity, denote by $U$ the domain of $F$ and by $V$ the range.
As $F(\mathcal{B}(F))$ has topological dimension at most $1$, we can find $x\in V$ and $\epsilon_1 >0$ so that $B(x,\epsilon_1) \subset V \setminus F(\mathcal{B}(F))$.

By shrinking $\epsilon_1$ if neccesary, we may assume that $F^{-1}(B(x,\epsilon_1))$ consists of $k=\operatorname{deg} F$ disjoint topological balls $E_1,\ldots, E_k$ and the restriction of $F$ to each $E_j$, for $j=1,\ldots, k$, is a homeomorphism. Let $u_j = F^{-1}(x) \cap E_j$ for $j=1,\ldots, k$. Choose
\[ \epsilon_2 < \frac{1}{2} \min \{ \epsilon_1 , \operatorname{dist}(u_1,\partial E_1),\ldots, \operatorname{dist}(u_k,\partial E_k) \} .\]
For $j=0,\ldots, k$, find affine maps $A_j$ such that $A_0$ maps $T^g$ into $B(x,\epsilon_2)$, and for $j=1,\ldots, k$, $A_j$ maps $T^g$ into $B(u_k,\epsilon_2)$. We view the collection $A_1(T^g),\ldots, A_k(T^g)$ as satellites to the first level $M_1$ of the defining sequence for $X_g$.

Recall the maps $\phi_1,\ldots, \phi_m$ from \eqref{eq:phim}. To these we add the maps $A_1,\ldots, A_k$ and relabel via $\xi_1,\ldots, \xi_{m+k}$ where $\xi_i = \phi_i$ if $i\in\{1,\dots,m\}$ and $\x_i = A_{i-m}$ if $i\in\{m+1,\dots,m+k\}$. As the images $\xi_i (T^g)$ and $\xi_j(T^g)$ are pairwise disjoint for $i\neq j$, we may define the Cantor set 
\[ Y_g = \bigcap_{n=1}^{\infty} \bigcup_{j_1,\ldots, j_n \in \{1,\ldots,m+k\} } \xi_{j_1} \circ \cdots \circ \xi_{j_n} (T^g).\]
As $Y_g$ contains $X_g$, its genus is at least $g$, and as the description above includes a defining sequence consisting of genus $g$ tori, the genus of $Y_g$ is also $g$.

Clearly the elements of $Y_g$ that are also contained in $X_g$ have local genus equal to $g$. To see that some elements of $Y_g$ have local genus zero, we observe that we can construct a different defining sequence for $Y_g$. The first level $\widetilde{M}_1$ of this new defining sequence for $Y_g$ is the first level of the defining sequence $M_1$ for $X_g$ together with the balls $B(u_1,\epsilon_2),\ldots, B(u_k,\epsilon_2)$. This yields the alternate description
\[ Y_g = \bigcap_{n=1}^{\infty} \bigcup_{j_1,\ldots, j_n \in \{1,\ldots,m+k\} } \xi_{j_1} \circ \cdots \circ \xi_{j_n} (\widetilde{M}_1).\]
If we let $y$ be the unique point in
\[\bigcap_{n=1}^{\infty} \xi_{m+k}^n(B(u_k,\epsilon_2)),\]
then $y\in Y_g$ and $\widetilde{M}_i^y$ is a ball for each level $i\geq 1$. We conclude that the local genus of $y$ is equal to $0$ and, assuming for the moment that $Y_g$ can be realized as a Julia set, Theorem \ref{thm:4} shows that there is a dense subset of $Y_g$ with local genus $0$.

Finally, we have to show that $Y_g$ can be realized as a Cantor Julia set. We modify the map $F$ above to $\widetilde{F}$ as follows. 
\begin{itemize}
\item On $U \setminus  \bigcup_{j=1}^k E_j$ we set $\widetilde{F} = F$. 
\item For each $j=1,\ldots, k$, we redefine $F$ on $\overline{B}(u_j,\epsilon_2)\setminus A_j(T^g)$ to be an isometry onto $\overline{B}(x,\epsilon_2) \setminus A_0(T^g)$. 
\item For $j=1,\ldots, k$, we use the bi-Lipschitz version of the Annulus Theorem \cite[Theorem 3.17]{TV} to extend $\widetilde{F}$ to a bi-Lipschitz map from $E_j \setminus \overline{B}(u_j,\epsilon_2)$ to $B(x,\epsilon_1) \setminus \overline{B}(x,\epsilon_2)$.
\end{itemize}

This yields a BLD map 
\[ \widetilde{F} : U \setminus \bigcup_{j=1}^k A_j(T^g) \to V \setminus A_0(T^g).\]
The image is the ball $\overline{B}(0,2C_{g,g})$ with two similar unlinked genus $g$ solid tori removed. By applying an auxiliary bi-Lipschitz map to $\overline{B}(0,2C_{g,g})$, we may obtain the images of the two removed tori are symmetric under the involution $\iota(x,y,z) =( x,-y,-z)$. The corresponding winding map $q_{\iota}$ is a degree $2$ sense preserving map which identifies the two tori removed from $\overline{B}(0,2C_{g,g})$. Proceeding as in the proof of Proposition \ref{prop:1}, and by applying further bi-Lipschitz deformations if necessary, we obtain the BLD map
\[ q_{\iota}\circ \widetilde{F} : U \setminus \bigcup_{j=1}^k A_j(T^g) \to \overline{B}(0,2C_{g,g}) \setminus \int(T^g).\]
The construction of the UQR map nows proceeds almost identically to the construction from Section \ref{subsec:uqr}. The only difference is that the UQR power map in a neighborhood of infinity has degree $2^{\lceil \log_2{g}\rceil + 1}4N_g$.
\end{proof}

\appendix

\section{A bi-Lipschitz deformation}\label{ap:ext}

Here we prove the existence of the map $h: U_2 \to V_2$ in Lemma \ref{lem:folding}. {For the rest of this appendix, we write $N=N_g$ where $N_g$ is the constant in \eqref{eq:N}. Recall that $T^g_{k+1,k+1,i} \cap V_2$ if and only if $i\in\{\tfrac12(N+1),\dots, \tfrac32(N+1)\}$. }

The construction of $h$ is based on two results. The first is an extension theorem of V\"ais\"al\"a.

\begin{theorem}[{\cite[Corollary 5.20]{Vais}}]\label{thm:Vaisext}
Let $n\geq 2$ and $\Sigma \subset \R^n$ be a compact PL manifold of dimension $n$ or $n-1$ with or without boundary. Then there exist $L,L' >1$ depending on $\Sigma$, such that every $L$-bi-Lipschitz embedding $F: \Sigma \to \R^n$ extends to an $L'$-bi-Lipschitz map $F:\R^n \to \R^n$.
\end{theorem}

Given sets $X,Y \subset \R^n$ we say that $H = \{H_t : X \to \R^3\}_{t\in[0,1]}$ is a \emph{bi-Lipschitz deformation of $X$ onto $Y$} if
\begin{enumerate}
\item for each $t\in [0,1]$, $H_t$ is a bi-Lipschitz map,
\item $H_0 = \text{Id}|X$ and  $H_1$ is a bi-Lipschitz homeomorphism of $X$ onto $Y$,
\item for any $\e>0$ and any $t\in [0,1]$ there exists $\delta>0$ such that for all $s \in [0,1]$ with $|s-t|<\delta$ we have $H_s \circ H^{-1}_t$ is $(1+\e)$-bi-Lipschitz.
\end{enumerate}

The second ingredient in the construction of $h$ is the following lemma. 

\begin{lemma}\label{lem:deform}
There exists a bi-Lipschitz deformation $H$ of $\partial V_2$ onto $\partial U_2$ {such that
\begin{enumerate}
\item for each $t\in [0,1]$, $H_t | \partial V_2 \cap \partial V_1$ is the identity,
\item for each $j \in \{\tfrac12(N+1),\dots, \tfrac32(N+1)\}$, $H_1|\partial T^g_{2k+1,k+1,j}\cap \partial V_2$ is an isometry mapping $\partial T^g_{2k+1,k+1,j}\cap \partial V_2$ onto $\partial T^g_{k+1,k+1,j}\cap \partial U_2$,
\item $H_1|\partial V_2\cap \partial V_3$ is a counterclockwise $\pi/2$-radians rotation with respect to the line $\{(x,y,z) : x=C_{2k+1,k}+1, y=\frac12\}$ that maps $\partial V_2\cap \partial V_3$ onto $\partial U_3\cap\partial U_2$.
\end{enumerate}}
\end{lemma}

Assuming we have constructed $H$, we proceed as follows. By Theorem \ref{thm:Vaisext}, for each $t \in [0,1]$, there exist constants $L_t, L_t' >1$ such that any $L_t$-bi-Lipschitz map $f: H_t(\partial V_2) \to \R^{3}$ has an $L_t'$-bi-Lipschitz extension $F: \R^{3} \to \R^{3}$. For all $t \in [0,1]$, there is an open interval $\Delta_t$ such that for all $s \in \Delta_t$, $H_s \circ H^{-1}_t$ is $L_t$-bi-Lipschitz. By compactness, we can cover $[0,1]$ with finitely many intervals $\{\Delta_{t_j}\}_{j=1}^l$, where $0 = t_0 < t_1 < \cdots < t_l = 1$ and $\Delta_{t_{j-1}}\cap \Delta_{t_j} \neq \emptyset$. For each $j=1,\dots,l$ set $a_{2j} = t_j$ and $a_{2j-1} \in \Delta_{t_{j-1}}\cap\Delta_{t_j}$. Then, each $H_{a_{j-1}}\circ H^{-1}_{a_{j}}$ extends to a bi-Lipschitz map $G_{a_{j-1}a_{j}}: \R^{3} \to \R^{3}$. Hence, the map
\[ G_{a_{2l-1}a_{2l}}\circ \cdots \circ G_{a_0a_1}\] 
is a bi-Lipschitz self-map of $\R^3$ that maps $V_2$ onto $U_2$ and its inverse is the desired map $h$.

\begin{proof}[{Proof of Lemma \ref{lem:deform}}]
The construction of $H$ is done in 3 steps.

Let $D_1=(\partial V_2 \cap  \{x = C_{k+1,k}+\frac12\})$, let $w_1$ be the center of $D_1$, and let $C_1$ be the {outermost boundary square} of $D_1$ (which is in the common boundary of $V_2$ and $V_1$). Similarly, let $D_3=\partial V_2 \cap (\{y = \frac12\}$, let $w_3$ be the center of $D_3$, and let $C_3$ be the {outermost boundary square} (which is in the common boundary of $V_2$ and $V_3$). Finally, let $w_2$ be the upper right corner point of the core curve $\Sigma^{k+1}$ and let $C_2$ be the boundary {rectangle} on $\partial V_2$ centered at $w_2$. See the left figure in Figure \ref{fig:deform2} for the projections on the $xy$-plane.

For the first step, we decompose $\partial V_2 = S_1\cup S_2\cup S_3\cup S_4$ where 
\begin{enumerate}
\item $S_1 = D_1 \cup \bigcup_{i=\frac12(N+1)}^{N}\partial T^g_{k+1,k+1,i}\cap \partial V_2$,
\item $S_2 = \partial T^g_{k+1,k+1,N+1}$ is the {boundary of the} upper right solid $g$-torus,
\item $S_3 = D_3 \cup \bigcup_{i=N+2}^{\frac32(N+1)}\partial T^g_{k+1,k+1,i}$
\item $S_4 = \partial V_2 \setminus (S_1\cup S_2\cup S_3)$ is a "crooked {square} cylinder".
\end{enumerate}

Set $l=(13\sqrt{2}-10)(20N)^{-1}$. Define the bi-Lipschitz deformation 
\[ H^{(1)} = \{ H^{(1)}_t: \partial V_2 \to \R^3\}_{t\in [0,1]}\]
such that 
\begin{itemize}
\item $H^{(1)}_t | S_1$ is the identity,
\item $H^{(1)}_t |S_2$ is a translation by $lt$ in the $y$ direction towards the negatives,
\item $H^{(1)}_t | S_3$ is a translation by $2lt$ in the $y$ direction towards the negatives,
\item $H^{(1)}_t | C_2$ is the identity,
\item $H^{(1)}_t | S_4$ is a linear interpolation of the maps $H^{(1)}_t | C_1$, $H^{(1)}_t | C_2$, and $H^{(1)}_t | C_3$.
\end{itemize}

\begin{figure}
\includegraphics[width=\textwidth]{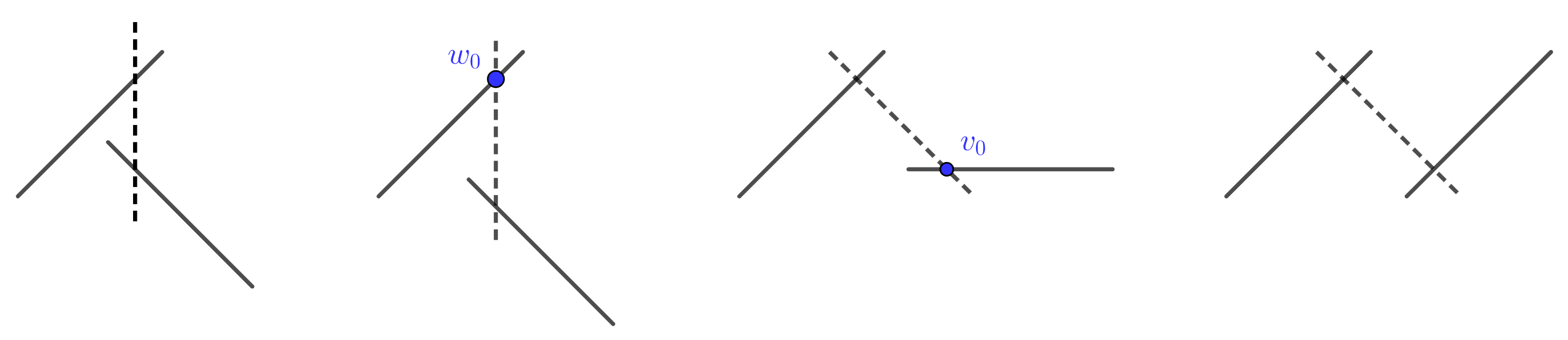}
\caption{The projections on the $xy$-plane of {the core curves of the $g$-tori bounded by $\partial T^g_{k+1,N_i}$, $i\in\{0,1,2\}$ (in the first figure), $H^{(1)}_1(\partial T^{g}_{k+1,N+i})$, $i\in\{0,1,2\}$ (in the second figure), $H^{(2)}_1\circ H^{(1)}_1( \partial T^{g}_{k+1,N+1})$, $i\in\{0,1,2\}$ (in the third figure), and $H^{(3)}_1\circ H^{(2)}_1\circ H^{(1)}_1(\partial T^{g}_{k+1,N+2})$ $i\in\{0,1,2\}$ (in the fourth figure).}}
\label{fig:deform1}
\end{figure}

\begin{figure}
\includegraphics[width=\textwidth]{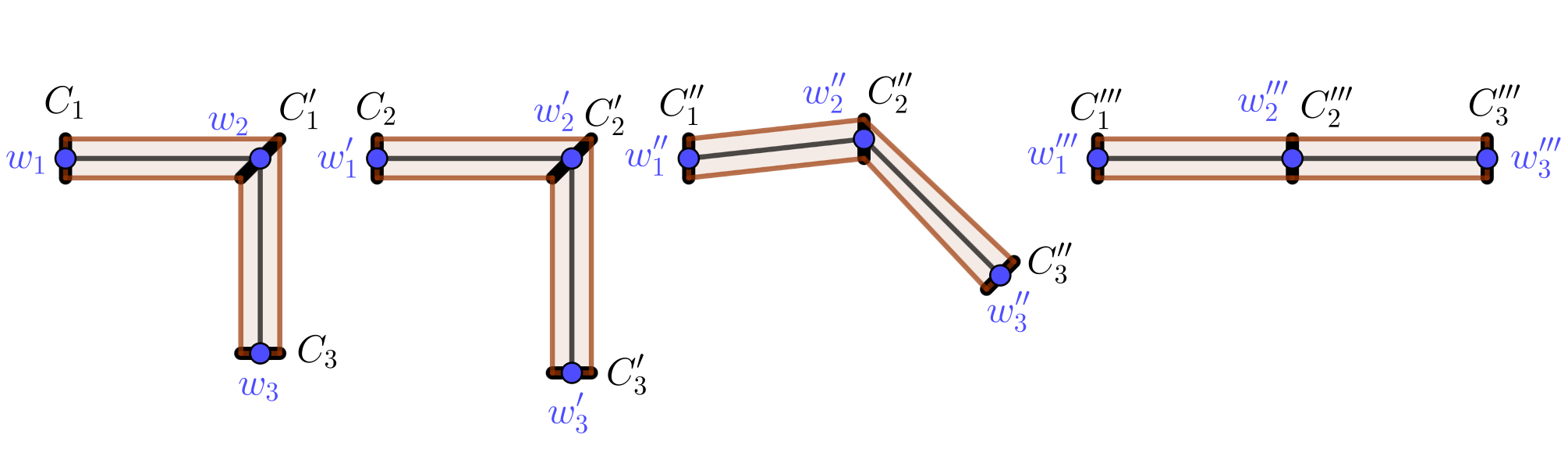}
\caption{The projections on the $xy$-plane of $S_4$ (first figure), $S_4'$ (second figure), $S_4''$ (third figure), $H^{(3)}(S_4'')$ (fourth figure).}
\label{fig:deform2}
\end{figure}

For the second step, set $S_i' = H^{(1)}_1(S_i)$ for $i=1,2,3,4$, set $w_i' = H^{(1)}_1(w_i)$ for $i=1,2,3$, and set $C_i' = H^{(1)}_1(C_i)$ for $i=1,2,3$. We define one more auxiliary point. Let $\pi:\R^3 \to \{z=0\}$ be the projection on the $xy$-plane and let $w_0 $ be the unique point in the intersection $\pi({\tau_1})\cap \pi({\tau_2})$ {where $\tau_1,\tau_2$ are the core curves of the solid $g$-tori bounded by $S_1',S_2'$, respectively.}

Let $R_t$ be the counterclockwise rotation map by $t\pi/4$ with respect to the line $\pi^{-1}(\{w_0\})$. We also use a bi-Lipschitz deformation that takes a {rectangle} to a {square} of the same center as the {rectangle} and side-length equal to the length of the shorter side of the rectangle. In particular {if $a,r>0$, then the map
\begin{equation}\label{eq:straight} 
\begin{cases}
(\pm (r+a),y) \mapsto (\pm (r+a) + \mp at, x), &  y \in [-r,r]\\
(x,\pm r) \mapsto (\frac{r+a-t}{r+a}x, \pm r), &x \in [-r-a,r+a]
\end{cases}
\end{equation}
deforms the rectangle $\partial ([-a-r,a+r]\times [-r,r])$ to the square $\partial [-r,r]^2$.}

Define now a bi-Lipschitz deformation $H^{(2)}= \{H^{(2)}_t:  H^{(1)}_1(\partial V_2)  \to \R^3\}_{t\in [0,1]}$ such that 
\begin{itemize}
\item $H^{(2)}_t | S_1'$ is the identity, 
\item $H^{(2)}_t|S_2'\cup S_3' =R_t$,
\item $H^{(2)}_t| C_2'$ is the composition of $R_t$ and a ``{rectangle-to-square}'' deformation as in \eqref{eq:straight},
\item $H^{(2)}_t | S_4'$ is a linear interpolation of the maps $H^{(2)}_t | C_1'$, $H^{(2)}_t | C_2'$, and $H^{(2)}_t | C_3'$.
\end{itemize}

For the third and final step, set $S_i'' = H^{(2)}_1(S_i')$ for $i=1,2,3,4$, set $w_i'' = H^{(2)}_1(w_i')$ for $i=1,2,3$, and set $C_i'' = H^{(2)}_1(C_i')$ for $i=1,2,3$.
We define one more auxiliary point. Let $v_0$ be the unique point in the intersection $\pi({\sigma_1})\cap \pi({\sigma_2})$, {where $\sigma_1,\sigma_2$ are the core curves of the solid $g$-tori bounded by $S_1'',S_2''$, respectively.}.
  
Let $R_t'$ be the counterclockwise rotation by $t\pi/4$ with respect to the line $\pi^{-1}\{v_0\}$. Define now a bi-Lipschitz deformation $H^{(3)} = \{ H^{(2)}_t: H^{(2)}_1(H^{(1)}_1(\partial V_2))  \to \R^3\}_{t\in [0,1]}$
such that 
\begin{itemize}
\item $H^{(3)}_t | S_1'' \cup S_2''$ is the identity, 
\item $H^{(3)}_t|S_3'' = R_t'$,
\item $H^{(3)}_t| C_2''$ is a translation with $H^{(3)}_t( C_2'')$ being centered at $R_t'(w_2'')$,
\item $H^{(3)}_t | S_4''$ is a linear interpolation of the maps $H^{(3)}_t | C_1''$, $H^{(3)}_t | C_2''$, and $H^{(3)}_t | C_3''$.
\end{itemize}

We finish by concatenating the deformations $H^{(1)}, H^{(2)}, H^{(3)}$ and obtain $H$ by defining $H_t = H^{(1)}_{3t}$ if $t \in [0,1/3]$, $H_t = H^{(2)}_{3t-1}$ if $t \in [1/3,2/3]$, and $H_t = H^{(3)}_{3t-2}$ if $t \in [2/3,1]$.
\end{proof} 

\section{UQR power maps}\label{ap:uqr}

It is well-known that there exist UQR analogues of power mappings in $\overline{\R^3}$ of degree $d^2$, where $d \geq 2$. These were first constructed by Mayer \cite{Mayer}. It is perhaps less well-known that other degrees may be achieved.

\begin{proposition}\label{prop:deg2}
There exists a UQR map $f:\overline{\R^3} \to \overline{\R^3}$ of degree $2$ with Julia set equal to the unit sphere $\S^2$.
\end{proposition}

\begin{proof}
If $x = (x_1,x_2,x_3) \in \R^3$, let $G$ be the discrete group of isometries in $\R^3$ generated by $g_1(x) = x + (1,0,0)$, $g_2(x)  = x + (0,1,0)$ and $g_3$ the rotation about the $x_3$-axis by angle $\pi$. Then there is a Zorich map $\mathcal{Z}$ which is strongly automorphic with respect to $G$ and which maps the plane $\{ x_3 = 0\}$ onto $\S^2$ (see, for example, \cite{FleMac} for more details on this). 

Let $A$ be the linear map which is a composition of a dilation with scaling factor $\sqrt{2}$ and a rotation by $\pi/4$ about the $x_3$-axis, that is,
\[ A(x_1,x_2,x_3) = (x_1-x_2,x_1+x_2,\sqrt{2}x_3).\]
We need $AGA^{-1} \subset G$ and it is sufficient to check that this is so on the generators. Clearly $Ag_3A^{-1} = g_3$. Next, by the linearity of $A$, we have
\[
Ag_1A^{-1}(x) = A\left ( A^{-1}(x) + (1,0,0) \right ) = x + A(1,0,0) = x+(1,1,0) = g_2(g_1(x))
\]
and hence $Ag_1A^{-1} \in G$. Finally,
\[
Ag_2A^{-1}(x) = A \left ( A^{-1}(x) + (0,1,0) \right ) = x + A(0,1,0) = x+ (-1,1,0) = g_2(g_1^{-1}(x))
\]
and hence $Ag_2A^{-1} \in G$. Thus by, for example, \cite[Theorem 3.4]{FleMac} we conclude that there is a UQR map $f:\R^3 \to \R^3$ which solves the Schr\"oder equation $f\circ \mathcal{Z} = \mathcal{Z} \circ A$. Proceeding as Mayer \cite{Mayer}, we see that $f$ extends over the point at infinity, the Julia set of $f$ is $\S^2$, and $f$ has degree $2$. 
\end{proof}

We may use the map from Proposition \ref{prop:deg2} to prove Theorem \ref{thm:degreeUQR}.

\begin{proof}[Proof of Theorem \ref{thm:degreeUQR}]
Let $f$ be the UQR map of degree $2$ from Proposition \ref{prop:deg2}, and denote by $P_d$, for $d\geq 2$, the UQR power maps constructed by Mayer \cite{Mayer} using the Zorich map $\mathcal{Z}$ above. Then following the argument in \cite[\textsection 5.1]{Fle}, $\mathcal{F}$ is a quasiregular semigroup, where $\mathcal{F}$ is generated by $f$ and the collection of $P_d$'s with Julia set $J(\mathcal{F}) = \S^2$. In particular, the map $f^n \circ P_d$ is UQR with the required properties.
\end{proof}

\section{Quasiregular uniformization of Cantor sets in $\R^3$ with controlled geometry}\label{ap:QRunif}

In \cite{DSbook}, David and Semmes showed that a metric space $X$ is quasisymmetrically homeomorphic to the standard Cantor set $\mathcal{C}$ if and only if it is compact, uniformly disconnected, and uniformly perfect. Later, MacManus \cite{MM2} proved a stronger uniformization result for Cantor sets contained in $\R^2$ by showing that for a compact set $X\subset \R^2$ there exists a quasisymmetric mapping $F:\R^2 \to \R^2$ with $F(\mathcal{C}) = X$ if and only if $X$ is uniformly perfect and uniformly disconnected. The improvement here is that the quasisymmetric homeomorphism can be in fact assumed to be defined on the ambient space $\R^2$ and not just $\mathcal{C}$.

MacManus' theorem is false in higher dimensions due to the existence of a wild (quasi-)self-similar Cantor sets in $\R^3$ and in $\R^4$ \cite{PW2}. To avoid such topological obstructions, usually an increase in dimension is required. In \cite{V3}, the third named author showed that for a compact set $X\subset \R^n$ there exists a quasisymmetric mapping $F:\R^{n+1} \to \R^{n+1}$ with $F(\mathcal{C}) = X$ if and only if $X$ is uniformly perfect and uniformly disconnected. 

In this appendix we prove a uniformization result for sets in $\R^3$. The first difference here is that the dimension of the space is not increased, and the second that quasisymmetry is replaced by the weaker quasiregularity. 

\begin{theorem}\label{thm:QR-unif}
For each $c\geq 1$, there exists $K\geq 1$ with the following property. If $X$ is a compact $c$-uniformly perfect and $c$-uniformly disconnected set in $\R^3$, then there exists a $K$-quasiregular map $f:\R^3 \to \R^3$ such that $f(X)=\mathcal{C}$.
\end{theorem}

The next lemma is well known but we give a short proof for the shake of completeness.

\begin{lemma}\label{lem:PLmfds}
Let $M$ be a compact connected PL 2-manifold in $\R^3$. There exists a PL branched covering $g : M \to \partial [0,1]^3$ of degree at least 3.
\end{lemma}

\begin{proof}
Let $\mathcal{S}_0$ be the boundary of $[0,1]^3$. For each $n\in\N$ let $\mathcal{S}_n$ be the boundary of
\[ \left( [0,2n+1]\times[-2,2] \setminus  \bigcup_{i=1}^n [2i-1,2i]\times [-1,1]\right) \times[-1,1] \subset \R^3.\]

Let $M$ be a compact connected PL 2-manifold in $\R^3$. By the Classification Theorem for surfaces in $\R^3$ \cite[Section 22]{Moise} there exists $n\in\N\cup\{0\}$ and a homeomorphism $M \to \mathcal{S}_n$. Since both $M$ and $\mathcal{S}_n$ are PL, there exists a PL homeomorphism $M \to \mathcal{S}_n$.

Fix $k\in\N$. We construct a degree 2 PL branched covering from $\mathcal{S}_k$ onto $\mathcal{S}_0$. Consider the two planes $P_1 = \{z=0\}$ and $P_2 = \{y=0\}$ which cut $\mathcal{S}_k$ into 4 PL disks $D_1,D_2,D_3, D_4$. Note that $P_1\cap P_2\cap \mathcal{S}_k$ contains exactly $2(k+1)$ many points $\{p_1,\dots,p_{2(k+1)}\}$ and these points are contained on the boundary curve of each $D_i$. We may assume that the pairs $(D_1,D_3)$ and $(D_2,D_4)$ intersect only on the points $\{p_1,\dots,p_{2(k+1)}\}$ while all other pairs intersect on boundary curves. Place now points $\{p_1',\dots,p_{2(k+1)}'\}$ on $(\partial [0,1]^2)\times \{\frac12\}$ oriented in the same way that points $\{p_1,\dots,p_{2(k+1)}\}$ are oriented on the boundary curve of $D_1$. There exists a 2-to-1 PL branched covering $\mathcal{S}_k \to \mathcal{S}_0$ that maps $p_i$ to $p_i'$ for each $i$, maps $D_1,D_2$ onto $\mathcal{S}_0 \cap \{z\geq \frac12\}$, and maps $D_3,D_4$ onto $\mathcal{S}_0 \cap \{z\leq \frac12\}$.

Thus, we have constructed a PL branched covering from $M$ onto $\mathcal{S}_0$ which has degree either 1 (if $n=0$) or 2 (if $n\geq 1$). To complete the proof, we compose this map with a degree 4 PL branched covering of $\mathcal{S}_0$ onto itself as in Figure \ref{fig:cubic-deg2}.
\end{proof}

\begin{figure}
    \centering
    \begin{minipage}{0.45\textwidth}
        \centering
        \includegraphics[width=0.7\textwidth]{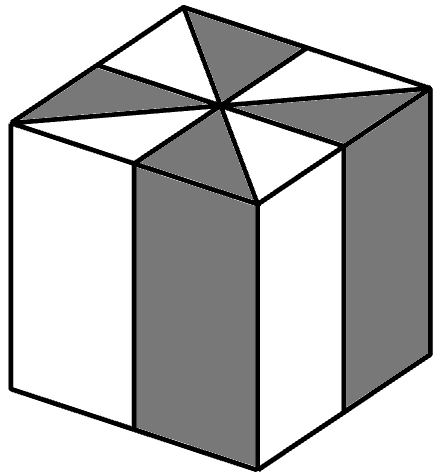} 
    \end{minipage}\hfill
    \begin{minipage}{0.45\textwidth}
        \centering
        \includegraphics[width=0.7\textwidth]{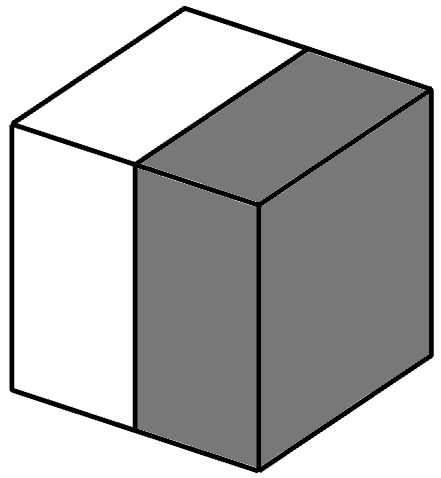} 
    \end{minipage}
    \caption{A degree 4 PL branched covering of $\mathcal{S}_0$ onto itself.}
    \label{fig:cubic-deg2}
\end{figure}

We are now ready to prove Theorem \ref{thm:QR-unif}.

\begin{proof}[{Proof of Theorem \ref{thm:QR-unif}}]
Let $X$ be a compact $c$-uniformly perfect and $c$-uniformly disconnected set in $\R^3$. Let $\mathcal{W}$, $\{\mathcal{M}_1,\dots,\mathcal{M}_l\}$, $\{\mathcal{N}_1,\dots,\mathcal{N}_q\}$, ${\bf i}:\mathcal{W} \to \{1,\dots,l\}$, ${\bf j}:\mathcal{W} \to \{1,\dots,q\}$, and $\{\phi_w\}_{w\in\mathcal{W}}$ be the dictionary, two finite collections of PL manifolds, functions, and similarities as in Lemma \ref{lem:finite defining sequence}. Following the arguments in the proof of \cite[Corollary 5.2]{BV}, we may further assume that $\mathcal{M}_{{\bf i}(\varepsilon)}$ is a cube.

Let $\zeta_1,\zeta_2 : \R^3 \to \R^3$ be the similarities $\zeta_1(p) = \frac13 p$ and $\zeta_2(p) = \frac13 p + (\frac23,0,0)$. Note that $\{\zeta_1,\zeta_2\}$ is an iterated function system with attractor the standard Cantor set $\mathcal{C}$. Given a word $u=i_1\cdots i_k \in \{1,2\}^*$ we define $\zeta_{u} = \zeta_{i_1}\circ\cdots\circ \zeta_{i_k}$. Conventionally, we set $\zeta_{\varepsilon}$ be the identity map. 

Set $Q = [-1/3,4/3]\times[-1,1]^2$. For all $u\in\{1,2\}^*$ and $i\in\{1,2\}$, $\zeta_{ui}(Q) \subset \zeta_{u}(Q)$, $\diam{\zeta_{wi}(Q)} = \frac13\diam{\zeta_w(Q)}$, and 
\begin{align*}
\dist(\zeta_{ui}(Q),\partial \zeta_{u}(Q)) &\geq \tfrac1{15}\diam{\zeta_u(Q)}\\
\dist(\zeta_{u1}(Q),\zeta_{u2}(Q)) &\geq \tfrac1{15}\diam{\zeta_u(Q)}.
\end{align*}

For each $j\in\{1,\dots,q\}$ we define a PL 3-manifold with boundary $\mathcal{Q}_j$ as follows. Fix $j\in \{1,\dots,q\}$ and denote by $m_j$ the number of boundary components of $\mathcal{N}_j$. Recall that $m_j \geq 3$. Let $k_j$ be the largest integer such that $2^{k_j}+1 \leq m_j$ and consider the following two possible cases. 
\begin{enumerate}
\item If $m_j = 2^{k_j}+1$, then set $\mathcal{Q}_j = Q \setminus \bigcup_{w\in \{1,2\}^{k_j}}\zeta_{w}(Q)$.
\item If $2^{k_j} + 1<m_j$, by maximality of $k_j$ we also have $2^{k_j +1}+1>m_j$. Fix $v_{j,1},\dots,v_{j,n_j} \in \{1,2\}^{k_j}$ with $n_j=m_j - 2^{k_j}-1$. Define now
\[ \mathcal{Q}_j = Q \setminus \left( \bigcup_{u \in \{1,2\}^*\setminus \{v_{j,1},\dots,v_{j,n_j}\}} \zeta_u(Q)\right) \setminus \left(\bigcup_{i=1}^{n_j} \zeta_{v_{j,i}1}(Q) \cup \zeta_{v_{j,i}2}(Q) \right).\]
\end{enumerate}
Note that for all $j\in\{1,\dots,q\}$, $\partial\mathcal{Q}_j$ has exactly as many components as $\partial \mathcal{N}_j$ does.

We define a function ${\bf u}:\mathcal{W} \to \{1,2\}^*$ as follows. Define ${\bf u}(\varepsilon) = \varepsilon$. Suppose now that for some $w\in \mathcal{W}$ we have defined ${\bf u}(w)$. Recall from Lemma \ref{lem:finite defining sequence}(P6) that $\mathcal{N}_{{\bf j}(w)}$ has $N_w+1$ many components. We consider two possible cases. 
\begin{enumerate}
\item If $2^{k_{{\bf j}(w)}}=N_w$, then set $\{{\bf u}(wi) : 1\leq i \leq N_w\} = \{{\bf u}(w)v:v\in \{1,2\}^{k_{{\bf j}(w)}} \}$.
\item If $2^{k_{{\bf j}(w)}}<N_w$, then set
\begin{align*} 
\{{\bf u}(wi) : 1\leq i \leq N_w\} &= \left\{{\bf u}(w)v : v\in \{1,2\}^{k_{{\bf j}(w)}}\setminus\{v_{{\bf j}(w),1},\dots,v_{{\bf j}(w),n_{{\bf j}(w)}}\}\right\}\\ 
&\qquad \cup \bigcup_{i=1}^{n_{{\bf j}(w)}}\{{\bf u}(w)v_{{\bf j}(w),i}1,{\bf u}(w)v_{{\bf j}(w),i}2\}.
\end{align*} 
\end{enumerate}

It is clear that for any $w\in \mathcal{W}$ there exists $u\in \{1,2\}^*$ (for example take $u= {\bf u}(w)$) such that $\zeta_u(Q) \subset \zeta_{{\bf u}(w)}$. We claim now that for any $u\in \{1,2\}^*$ there exists $w\in \mathcal{W}$ such that $\zeta_{{\bf u}(w)} \subset \zeta_u(Q)$. Assuming the claim, it follows that
\begin{equation}\label{eq:newdef}
\mathcal{C} = \bigcap_{n=1}^{\infty}\bigcup_{u\in \{1,2\}^n}\zeta_{u}(Q) = \bigcap_{n=1}^{\infty}\bigcup_{\substack{w\in \mathcal{W}\\|w|=n}}\zeta_{{\bf u}(w)}(Q).  
\end{equation}
To prove the claim, fix $u\in \{1,2\}^*$ and let $\mathcal{U}_u = \{w\in \mathcal{W} : \zeta_u(Q) \subset \zeta_{{\bf u}(w)}(Q)\}$. Clearly, $\varepsilon \in \mathcal{U}_u$. Let $w\in \mathcal{U}_u$ be an element of maximal length. Assume first that $N_w = 2^{k_{{\bf j}(w)}}$. By maximality of $|w|$, we have that $|u| > |{\bf u}(w)|+ k_{{\bf j}(w)}$ and it follows that there exists $i\in \{1,\dots,N_w\}$ such that $\zeta_{{\bf u}(wi)}(Q) \subset \zeta_u(Q)$. Assume now that $N_w > 2^{k_{{\bf j}(w)}}$. If $|u| \leq |{\bf u}(w)| + k_{{\bf j}(w)}$, then from design of ${\bf u}$, there exists $i\in \{1,\dots,N_w\}$ such that $\zeta_{{\bf u}(wi)}(Q) \subset \zeta_u(Q)$. If $|u| > |{\bf u}(w)| + k_{{\bf j}(w)}$, either $\zeta_{u}(Q) \subset \zeta_{{\bf u}(w)v}(Q)$ for some $v \in \{1,2\}^{k_{{\bf j}(w)}}\setminus\{v_{{\bf j}(w),1},\dots,v_{{\bf j}(w),n_{{\bf j}(w)}}\}$ which is impossible by maximality of $|w|$, or $\zeta_{u}(Q) \subset \zeta_{{\bf u}(w)v_{i,n_{{\bf j}(w)}}}(Q)$ for some $i\in\{1,\dots,n_{{\bf j}(w)}\}$ which is impossible for the same reason.

By Lemma \ref{lem:PLmfds}, for each $i\in\{1,\dots,l\}$, there exists a PL branched covering $g_i : \partial\mathcal{M}_{i} \to \partial Q$ of degree at least 3. By Theorem \ref{thm:BE}, for each $j\in\{1,\dots,q\}$ there exists a sense-preserving PL branched covering $G_j : \mathcal{N}_j \to \mathcal{Q}_j$.

Finally, we define a quasiregular map $f:\R^3 \to \R^3$ with $f(X)=\mathcal{C}$. Note that
\[\R^3 = (\R^3 \setminus \phi_{\varepsilon}(\mathcal{M}_{{\bf i}(\varepsilon)})) \cup \left( \bigcup_{w\in \mathcal{W}}\phi_{w}(\mathcal{N}_{{\bf j}(w)}) \right) \cup X.\]
Define $f$ so that
\begin{enumerate}
\item $f| \R^3 \setminus \phi_{\varepsilon}(\mathcal{M}_{{\bf i}(\varepsilon)}) :\R^3 \setminus \phi_{\varepsilon}(\mathcal{M}_{{\bf i}(\varepsilon)})  \to \R^3 \setminus Q$ is a PL bi-Lipschitz map with
\[ f\circ \phi_{\varepsilon} | \partial \mathcal{M}_{{\bf i}(\varepsilon)} = g_{{\bf i}(\varepsilon)},\]
\item for all $w \in \mathcal{W}$
\[ f| \phi_w(\mathcal{N}_{{\bf j}(w)}) : \phi_w(\mathcal{N}_{{\bf j}(w)}) \to \zeta_{{\bf u}(w)}(\mathcal{Q}_{{\bf j}(w)})\]  
with
\[f\circ \phi_{w} | \mathcal{N}_{{\bf j}(w)} = \zeta_{{\bf u}(w)} \circ G_{{\bf j}(w)},\]
\item if $x \in X$ and $x$ is the unique point in $\bigcap_{k=1}^{\infty}\phi_{i_1\cdots i_k}(\mathcal{M}_{{\bf i}(i_1\cdots i_k)})$, then $f(x)$ is the unique point in $\bigcap_{k=1}^{\infty}\zeta_{{\bf u}(i_1\cdots i_k)}(Q)$.
\end{enumerate}

It is easy to see that $f$ is a quasiregular map since, up to a set of $\sigma$-finite $\mathcal{H}^2$-measure, the map is a PL branched covering made up of a finite collection $\{G_{{\bf j}(w)}\}_{w\in \mathcal{W}}$ of PL branched coverings. Moreover, by \eqref{eq:newdef}, $f(X) = \mathcal{C}$.
\end{proof}

\bibliography{genus}
\bibliographystyle{amsbeta}

\end{document}